\documentclass[12pt,a4paper]{amsart}

\usepackage{amssymb, amstext, amscd, amsmath, color}

\usepackage{url}
\usepackage{tikz}
\usepackage{epstopdf}
\usepackage{amsfonts}
\usepackage{amsthm}
\usepackage{amsfonts}
\usepackage{array}
\usepackage{epsfig}
\usepackage{eucal}
\usepackage{latexsym}
\usepackage{mathrsfs}
\usepackage{textcomp}
\usepackage{verbatim}
\usepackage{setspace}
\usepackage{enumerate}
\usepackage{hyperref}

\newtheorem{thm}{Theorem}[section]

\newtheorem{cor}[thm]{Corollary}

\newtheorem{defn}[thm]{Definition}

\newtheorem{lem}[thm]{Lemma}

\newtheorem{prop}[thm]{Proposition}
\newtheorem{rem}[thm]{Remark}

\numberwithin{equation}{section}

\renewcommand{\H}{{\mathcal{H}}}

\renewcommand{\O}{{\mathcal{O}}}

  \newcommand{\U}{{\mathcal{U}}}

\newcommand{\rank}{\operatorname{rank}}

\newcommand{\spann}{\operatorname{span}}
\newcommand{\Iso}{\operatorname{Isom}}
\newcommand{\reg}{\operatorname{Reg}}
\newcommand{\sreg}{\operatorname{Str}}
\newcommand{\creg}{\operatorname{Com}}
\newcommand{\csreg}{\operatorname{ComStr}}
\newcommand{\grig}{\operatorname{GRig}}

\newcommand{\Isolin}{\operatorname{Isom^{Lin}}}

\usepackage{blkarray}

\tikzstyle{vertex}=[circle, draw, fill=black, inner sep=0pt, minimum size=3pt]
\tikzstyle{fadedvertex}=[vertex, fill=black!30]
\tikzstyle{edge}=[line width=1pt]
\tikzstyle{fadededge}=[edge,color=black!30]

\begin{document}

\title[Generalised rigid body motions in non-Euclidean planes]{Generalised rigid body motions in non-Euclidean planes with applications to global rigidity}
\author[Sean Dewar]{Sean Dewar}
\address{Johann Radon Institute\\ Altenberger Strasse 69\\ 4040\\ Linz\\
Austria}
\email{sean.dewar@ricam.oeaw.ac.at}
\author[Anthony Nixon]{Anthony Nixon}
\address{Dept.\ Math.\ Stats.\\ Lancaster University\\
Lancaster LA1 4YF \\U.K.}
\email{a.nixon@lancaster.ac.uk}

\begin{abstract}
A bar-joint framework $(G,p)$ in a (non-Euclidean) real normed plane $X$ is the combination of a finite, simple graph $G$ and a placement $p$ of the vertices in $X$. A framework $(G,p)$ is globally rigid in $X$ if every other framework $(G,q)$ in $X$ with the same edge lengths as $(G,p)$ arises from an isometry of $X$. The weaker property of local rigidity in normed planes (where only $(G,q)$ within a neighbourhood of $(G,p)$ are considered) has been studied by several researchers over the last 5 years after being introduced by Kitson and Power for $\ell_p$-norms. However global rigidity is an unexplored area for general normed spaces, despite being intensely studied in the Euclidean context by many groups over the last 40 years. In order to understand global rigidity in $X$, we introduce new generalised rigid body motions in normed planes where the norm is determined by an analytic function. This theory allows us to deduce several geometric and combinatorial results concerning the global rigidity of bar-joint frameworks in $X$.
\end{abstract}

\date{\today}
\maketitle

\setcounter{tocdepth}{1}
\tableofcontents

\section{Introduction}

Everyone understands rotations and reflections in Euclidean spaces. However under non-Euclidean norms these become subtle notions since lengths depend on direction. Cook, Lovett and Morgan \cite{clm} studied rotations in normed planes, and in particular observed that a cross braced rhombus cannot be fully rotated. Motivated both by their work and by potential applications to the study of global rigidity we develop generalised rigid body motions in non-Euclidean (typically analytic) normed planes.

A bar-joint framework $(G,p)$ in $\mathbb{R}^d$ is the combination of a graph $G=(V,E)$ and a map $p:V\rightarrow \mathbb{R}^d$. A fundamental question in the field of rigidity theory is when a given framework $(G,p)$ is unique. That is, that every other framework $(G,q)$ in the same dimension and on the same graph arises from $(G,p)$ by an isometry of $\mathbb{R}^d$. This is the global rigidity problem (see \cite{con,C2005,GHT,hendrickson,J&J} inter alia). We will use our results on generalised rigid body motions to initiate a study of global rigidity in normed planes. 

Some of our results will apply in general normed planes, however our main results will apply in the setting of analytic normed planes. A (finite dimensional real) normed space $X$ is \emph{analytic} if the restriction of the norm of $X$ to $X \setminus \{0\}$ is an analytic function.
Any such normed space is \emph{smooth} (i.e.~the norm is differentiable at each point in $X \setminus \{0\}$) and \emph{strictly convex} (i.e.~$\|x + y\|<\|x\|+\|y\|$ for all linearly independent $x,y\in X$). In this article we will always assume our normed spaces are non-Euclidean unless explicitly stated otherwise.

While global rigidity in normed planes seems to be an unexplored area of research, the situation for local rigidity (where the uniqueness of $(G,p)$ is only required in a neighbourhood of $p$) is different. This problem was first considered by  Kitson and Power \cite{kit-pow-1} who began a research program in this area by proving a combinatorial description of when `most' frameworks $(G,p)$ in the normed plane $\ell_q^2$ ($1<q<\infty$ and $q\neq 2$) are rigid. Note that this generalised the result for the Euclidean case obtained by Pollaczek-Geiringer \cite{PG} and popularised by Laman \cite{laman}. Kitson and Power's work led, for example, to Dewar's analysis of more general normed planes \cite{dew2,dewphd,D19} and Levene and Kitson's extension to matrix norms \cite{LK}.

Additionally, a number of papers have considered the similar sounding problem of rigidity and global rigidity under non-Euclidean metrics, such as hyperbolic and Minkowski metrics \cite{GTfields,NW, SalW}. However these contexts are sufficiently similar to the Euclidean case that they invite the same combinatorial techniques and at the level of infinitesimal rigidity there is an elegant projective invariance \cite{NSW,SalW}. In the case of normed planes these conveniences are not available and the nature of the graphs that are typically rigid or globally rigid changes markedly. Moreover since we no longer have quadratic constraints we also do not have the luxury of having an obviously accessible equilibrium stress matrix approach (as in \cite{con,C2005,GHT}). Instead we will use our newly gained insight into geometry in analytic normed planes through generalised rigid body motions to deduce results on global rigidity. In particular we will show that every $k$-lateration graph (on at least 5 vertices and with $k\geq 3$) is globally rigid in any analytic normed plane. We also explicitly construct globally rigid graphs on $n$ vertices in analytic normed planes, for all $n\geq 6$, that are not globally rigid in the Euclidean case.

We conclude the introduction with a brief outline of what follows. In Section \ref{sec:prelim} we introduce the analytic background we shall use. Then Section \ref{sec:rigid} recaps and extends the theory of rigidity in normed spaces. In Section \ref{sec:reflect} we study reflections and rotations in normed planes in detail obtaining our key technical results on rigid body motions. Some of the more technical proofs are deferred to Section \ref{sec:tech} after we deduce our results on globally rigid graphs in Section \ref{sec:global}. We finish in Section \ref{sec:conc} with some avenues for further exploration.

\section{Background}
\label{sec:prelim}

\subsection{Negligible sets}

The concept of a set being ``negligible'' can be described in many (often non-equivalent) ways. There are two important properties that negligible sets should satisfy; (i) the complement of the set is dense, and (ii) the countable union of negligible sets will remain negligible.
There are two types of negligible sets we shall deal with throughout.

\begin{defn}
    A set $A \subset \mathbb{R}^d$ is a \emph{null set} if its Lebesgue measure is 0.
    A set $B \subset \mathbb{R}^d$ is a \emph{meagre set} if it is the countable union of nowhere dense\footnote{Recall that a set is \emph{nowhere dense} if its closure has empty interior.} sets.
    If $S$ is a subset of an open set $O \subset \mathbb{R}^d$,
    then $S$ is a \emph{conull} (respectively, \emph{comeagre}) subset of $O$ if $O \setminus S$ is a null (respectively, meagre) set.
    (If the set $O$ is well understood, we will simply refer to a set being conull/comeagre.)
\end{defn}

Both null sets and meagre sets satisfy the two desired properties for negligible sets, however the two concepts are distinct.
Indeed we can construct sets that are null and comeagre (and hence also sets that are conull and meagre).
We show this with the following subset of $\mathbb{R}$.
Label the rational numbers $x_1,x_2,\ldots$ and define the open intervals $U_i(n)=(x_i - \frac{2^{-i}}{n}, x_i + \frac{2^{-i}}{n})$.
We now define the sets  $V_n := \bigcup_{i \in \mathbb{N}} U_i(n)$ and their intersection $V := \bigcap_{n \in \mathbb{N}} V_n$.
The set $V$ is a null set as each set $V_n$ has Lebesgue measure $2/n$,
but it is also comeagre since each set $V_n$ is also an open dense set.
Hence $V$ is null and comeagre, and $\mathbb{R} \setminus V$ is conull and meagre.

When the set is also restricted to be open/closed, we do obtain a hierarchy.
This follows since every closed null set is meagre, and every open conull set is comeagre (as every open dense set is comeagre). Also, the only open null/meagre set is the empty set, and the only closed conull/comeagre subset of an open set is itself.

As all (finite dimensional real) normed spaces are isomorphic to $\mathbb{R}^d$,
we can define (co)meagre sets and (co)null sets for any normed space.

\subsection{Real analytic functions}

For every $k = (k_1,\ldots,k_d) \in (\mathbb{N}\cup\{0\})^d$ and every $x = (x_1,\ldots,x_d) \in \mathbb{R}^d$ we define $x^k := x_1^{k_1} \cdot \ldots \cdot x_d^{k_d}$.
Let $U \subset \mathbb{R}^d$ be open and let $\|\cdot\|_2$ be the standard Euclidean norm for $\mathbb{R}^d$.
A function $f:U \rightarrow \mathbb{R}$ is \emph{analytic} if, for every $z \in U$, there exists $r > 0$ and $a_k \in \mathbb{R}$ for each $k \in (\mathbb{N}\cup\{0\})^d$, such that if $\|x-z\|_2 <r$ then; the point $x$ lies in $U$, the series $\sum_{k \in (\mathbb{N}\cup\{0\})^d} a_k(x - z)^k$ is absolutely convergent, and
\begin{align*}
    f(x) := \sum_{k \in (\mathbb{N}\cup\{0\})^d} a_k(x - z)^k.
\end{align*}
A map $g: U \subseteq \mathbb{R}^d \rightarrow \mathbb{R}^n$ is \emph{analytic} if $U$ is open and each coordinate function $g_1,\ldots,g_n$ is an analytic function.
It can be quickly checked that every analytic function/map is smooth,
and every derivative will also be analytic.

We recall that for a map $f:A \rightarrow B$ and sets $U \subset A$ and $V \supset f(A)$,
we denote the map formed from $f$ by restricting the domain to $U$ by $f|_U$, the map formed from $f$ by restricting the codomain to $V$ by $f|^V$, and the map formed from $f$ by restricting both the domain and codomain by $f|_U^V$.
An analytic map $f: U \rightarrow V$ between open sets $U,V \subset \mathbb{R}^d$ is called an \emph{analytic diffeomorphism} if it has an inverse that is also an analytic map.
An analytic map $f: U \rightarrow V$ is called a \emph{local analytic diffeomorphism} if for every point $x \in U$ there exists an open neighbourhood $U' \subset U$ and an open neighbourhood $V'$ of $f(x)$ where the map $f|_{U'}^{V'}$ is an analytic diffeomorphism.
We can detect (local) analytic diffeomorphisms via rank of the derivative with the following result.

\begin{thm}\cite[Theorem 1.8.1]{analytic}
\label{thm:imt}
	Let $U,V$ be open subsets of $\mathbb{R}^d$.
	If $f:U \rightarrow V$ is an analytic map and $\rank df(x) = d$ for some $x \in U$,
	then there exists a neighbourhood $U' \subset U$ of $x$ so that $f|_{U'}$ is a local analytic diffeomorphism.
	Furthermore,
	if $f$ is injective and $\rank df(x) = d$ for all $x \in U$, $f$ is an analytic diffeomorphism.
\end{thm}

By using charts where the transition maps are analytic diffeomorphisms,
we can construct \emph{real analytic manifolds} in the usual way;
see for example \cite{manifold}.

Analytic functions generalise polynomials and the zero sets of analytic functions (denoted throughout by $Z(f)$) share many of the properties of algebraic sets.
The following result is folklore,
however we direct the interested reader to \cite{Mityagin15} for a simple proof.

\begin{prop}\label{p:anvar}
	Let $U \subset \mathbb{R}^d$ be an open connected subset and $f:U \rightarrow \mathbb{R}$ be analytic.
	If the zero set $Z(f)$ is not equal to $U$ then it is a closed null subset of $U$.
\end{prop}

We can improve upon Proposition \ref{p:anvar} significantly for non-zero analytic functions with a 1-dimensional domain by showing that they must have a locally finite zero set.

\begin{prop}\cite[Corollary 1.2.6]{analytic}\label{p:analytic1d}
	Let $U$ be an open interval and $f:U \rightarrow \mathbb{R}$ an analytic function.
	If there exists a convergent sequence $(x_n)_{n \in \mathbb{N}}$ in $U$ where $f(x_n)=0$ for all $n \in \mathbb{N}$ and $\lim_{n \rightarrow \infty} x_n \in U$,
	then $Z(f)=U$.
\end{prop}

We can also state a similar result for $d=2$, although we will first need to lay out some terminology.
We will use $\mathbb{T}$ to denote the circle group,
i.e. $\mathbb{T} = \mathbb{R}/ 2\pi \mathbb{Z}$, represented by the set $[0,2\pi)$.
To define the topology we will use the metric $d_{\mathbb{T}}$ with $d_{\mathbb{T}}(s,t) := \min\{ |s-t| , 2\pi - |s-t| \}$.
We also recall that a map between any two topological spaces is \emph{proper} if the preimage of any compact set is compact.

\begin{prop}\label{p:anvar2}
	Let $U \subset \mathbb{R}^2$ be an open connected subset and $f:U \rightarrow \mathbb{R}^2$ be analytic.
	Then one of three possibilities holds:
	\begin{enumerate}[(i)]
		\item $Z(f) \in \{ \emptyset ,U \}$.
		\item $Z(f)$ is countable.
		\item There exists a real analytic manifold $M$ that is the disjoint union of countably many copies of the sets $\mathbb{R}$ and $\mathbb{T}$, and there exists an analytic map $\phi : M \rightarrow U$ so that $\phi(M) = Z(f)$.
		Further,
		there exists a discrete set $M' \subset M$ so that $\rank d \phi(x) =1$\footnote{Here $d\phi(x)$ denotes the Jacobian derivative of the map $\phi$ evaluated at $x$.} for all $x \in M \setminus M'$.
	\end{enumerate}
\end{prop}

\begin{proof}
	Suppose $Z(f) \neq \emptyset$.
	By \cite[Theorem 5.4.8]{analytic},
	there exists a real analytic manifold $M$,
	a closed proper subset $M' \subset M$ and a proper analytic map $\phi : M \rightarrow U$ so that $\phi(M) = Z(f)$ and $d\phi(x)$ is injective for each $x \in M \setminus M'$.
	As $\dim U =2$ then $\dim M \in \{0,1,2\}$.
	If $\dim M = 0$ then $M$ is a countable set of points and $Z(f)$ is countable.
	If $\dim M=2$ then, by Theorem \ref{thm:imt},
	there exists an open set $O \subset U$ where $O \subset Z(f)$;
	hence by Proposition \ref{p:anvar},
	$Z(f)=U$.
	So $\dim M =1$.
	Now $M$ is the union of countably many sets $\{M_i : i \in I\}$ where each set $M_i$ is analytically diffeomorphic to either $\mathbb{R}$ or $\mathbb{T}$.
	Since $\rank d\phi(x) = 0$ if and only if $d \phi(x)\cdot d \phi(x) = 0$,
	the set $M'$ is the zero set of a non-zero analytic function with $1$-dimensional domain.
	The result now follows from Proposition \ref{p:analytic1d}.
\end{proof}

The next result is the analytic version of the implicit function theorem, and can be proven by combining the methods of \cite[Theorem 1.8.3]{analytic} and \cite[Theorem 2.5.7]{manifold}.

\begin{thm}
\label{thm:ifta}
	Let $U \subset \mathbb{R}^m$ and $V \in \mathbb{R}^n$ be open subsets.
	Suppose $f:U \times V \rightarrow \mathbb{R}^m$ is an analytic map and, for some $(x_0,y_0) \in U \times V$, the map
	\begin{align*}
		\tilde{f} : U \rightarrow \mathbb{R}^m, ~ x \mapsto f(x,y_0)
	\end{align*}
	has the property that $\rank d \tilde{f}(x_0) = m$.
	Then there exists a neighbourhood $V' \subset V$ of $y_0$, 
	a neighbourhood $W$ of $f(x_0,y_0)$ and a unique analytic map $g : V' \times W \rightarrow U$ such that for all $(y,w) \in V' \times W$,
	\begin{align*}
		f( g(y,w) , w) = w.
	\end{align*}
\end{thm}

\begin{rem}\label{rem:analyticfunctions}
    Since all $d$-dimensional normed spaces are isomorphic (as topological vector spaces) to $\mathbb{R}^d$,
    we can easily define analytic maps between normed spaces by choosing bases for our normed spaces.
    Hence, all of the results of this section will extend to general normed spaces.
\end{rem}

\subsection{Rigidity theoretic background}

For a set $V$ and a normed space $X$, a \emph{placement of $V$ in $X$} is any element $p \in X^{V}$.
We define the \emph{placement norm} to be the norm $\|\cdot \|_V$ of $X^{V}$ where $\|p\|_V := \max_{v \in V} \|p_v \|$.
Let $G = (V,E)$ be a (finite simple) graph with placement $p :V \rightarrow X$ in some (finite dimensional) normed space $X$, then the pair $(G,p)$ is called a \emph{framework in $X$}.
We define the set 
\begin{align*}
X_G := \left\{ p \in X^{V} : p_v \neq p_w \text{ for all } vw \in E \right\}.
\end{align*}
The set $X_G$ is an open conull subset of $X^{V}$; further, if $\dim X >1$ then $X_G$ is path-connected.
For a placement $p$ and affine map $g:X \rightarrow X$, we define $g\circ p := (g(p_v))_{v \in V}$.

A non-zero point $x$ in a normed space $X$ is a \emph{smooth point} if the norm is differentiable $x$;
equivalently,
$x$ is smooth if there exists a unique linear functional $\varphi_x : X \rightarrow \mathbb{R}$ where $\varphi_x(x)=\|x\|^2$ and $\|\varphi_x\|=\|x\|$.
If $x$ is smooth then $\varphi_x$ is the derivative of the map $z \mapsto \|z\|^2/2$ at the point $x$ (see \cite[Lemma 1]{kitschulze} for a proof).
A normed space $X$ is \emph{smooth} if every non-zero point of $X$ is smooth,
and \emph{strictly convex} if $\|tx + (1-t)y\| <2$ for all linearly independent $x,y$ with $\|x\|=\|y\|=1$ and all $0 < t < 1$.
A smooth normed space is strictly convex if and only if the map $x \mapsto \varphi_x$ is injective (equivalently, is a homeomorphism).

We define the \emph{rigidity map} to be
\begin{align*}
f_G : X^V  \rightarrow \mathbb{R}^E, ~ (x_v)_{v \in V} \mapsto \left( \frac{1}{2}\|x_v - x_w \|^2 \right)_{vw \in E}.
\end{align*}
A placement $p$ of $G$ is \emph{well-positioned} if $f_G$ is differentiable at $p$ and $p \in X_G$;
we note that for a graph $G=(V,E)$ with $|E|\geq 1$,
$X$ is smooth if and only if the set of well-positioned placements of $G$ is exactly the set $X_G$.
If $(G,p)$ is well-positioned, we define the derivative 
\begin{align*}
df_G(p) : X^V  \rightarrow \mathbb{R}^E, ~ (x_v)_{v \in V} \mapsto (\varphi_{p_v-p_w}(x_v - x_w ))_{vw \in E}.
\end{align*}
of $f_G$ at $p$ to be the \emph{rigidity operator of $(G,p)$}.

For a given choice of basis $\{e_i : i=1, \ldots,d\}$ for $X$, we may also define the \emph{rigidity matrix of $(G,p)$ in $X$} to be the $|E| \times d|V|$ real matrix $R(G,p)$ with entries $a_{e, (v,i)}$, where
\begin{align*}
a_{e,(v,i)} := 
\begin{cases}
\varphi_{v,w}(e_i), & \text{if } e = vw \in E(G) \\
0, & \text{otherwise}.
\end{cases}
\end{align*} 
If $X= \mathbb{R}^d$ we will use the standard basis, i.e.~$e_i$ is the vector with $1$ for its $i$-th coordinate and $0$ elsewhere.
We note immediately that $\rank R(G,p) = \rank df_G(p)$.

The sets of isometries and linear isometries of $X$ will be denoted by $\Iso (X)$ and $\Isolin (X)$ respectively.
As $\Iso(X)$ is a Lie group, we may define the tangent space at the identity map $\iota$ (denoted by $T_\iota \Iso (X)$).
The following implies that all isometries of a normed space are affine.

\begin{thm}\label{t:mazurulam}(see \cite[Theorem 3.1.2]{thompson})
	Let $X$ be a normed space,
	$U \subset X$ an open connected subset
	and $g :U \rightarrow X$ be an isometry.
	Then there exists an affine isometry $g' : X \rightarrow X$ where $g'|_U = g$.
\end{thm}

For a given placement $p$ we may define the \emph{orbit of $p$} and \emph{space of infinitesimal flexes of $p$} to be the respective sets
\begin{align*}
\mathcal{O}_p := \{ g\circ p : g \in \Iso (X) \} \mbox{ and }
\mathcal{T}(p) := \{ g\circ p : g \in T_{\iota} \Iso (X) \}.
\end{align*}
A placement $p$ is \emph{isometrically full} if and only if there exists a diffeomorphism between $\mathcal{O}_p$ and $\Iso (X)$,
and \emph{full} if and only if there exists a local diffeomorphism between $\mathcal{O}_p$ and $\Iso (X)$,
i.e., $\dim \mathcal{T}(p) = \dim \Iso (X)$.
Equivalently,
$p$ is isometrically full if for every isometry $g$ with $g \circ p = p$ we have that $g$ is the identity map,
and $p$ is full if there is a finite set of isometries $g$ with $g \circ p = p$ (see \cite[Proposition 3.9(ii), Theorem 3.15]{D19}).
Any isometrically full placement is full,
any placement $p$ where the set $\{p_v:v \in V\}$ affinely spans $X$ is isometrically full,
and any placement $p$ where the set $\{p_v:v \in V\}$ affinely spans an affine hyperplane of $X$ is full.

We say that two frameworks $(G,p)$ and $(G,q)$ in a normed space $X$ are \emph{equivalent} 
(denoted by $(G,p) \sim (G,p')$) if $f_G(p) = f_G(p')$,
and we say that they are \emph{congruent} (denoted by $p \sim p'$) if there exists an isometry $g \in \Iso (X)$ such that $p' = g \circ p$.

Let $(G,p)$ be a well-positioned framework in a normed space $X$.
We define a framework $(G,p)$ to be \emph{regular} if the rigidity operator $df_G(p)$ has maximal rank,
and \emph{strongly regular} if for all frameworks equivalent to $(G,q)$ are regular.
The placement $p$ is \emph{completely regular} if for all graphs $H$ with $V(H) := V$, $(H,p)$ is well-positioned and regular,
and \emph{completely strongly regular} if for all graphs $H$ with $V(H) := V$, $(H,p)$ is well-positioned and strongly regular.
We denote the set of regular placements of $G$ in $X$ by $\reg(G;X)$, the set of strongly regular placements of $G$ in $X$ by $\sreg(G;X)$, the set of completely regular placements of a set $V$ in $X$ by $\creg (V;X)$,
and the set of completely strongly regular placements of a set $V$ in $X$ by $\csreg (V;X)$.
We note that $(G,p)$ is strongly regular if and only if $f_G(p)$ is a regular value of the map $f_G$.

\begin{rem}
	If a framework is either strongly or completely regular then it is regular,
	and if a framework is completely strongly regular then it is both strongly and completely regular.
	However,
	completely regular and strongly regular are distinct,
	as we illustrate in Figure \ref{fig:notcomreg}.
\end{rem}

 \begin{figure}[htp]
\begin{center}
\begin{tikzpicture}[scale=.4]

\filldraw (0,0) circle (3pt)node[anchor=east]{};
\filldraw (-0.5,3.5) circle (3pt)node[anchor=east]{};
\filldraw (3.5,0) circle (3pt)node[anchor=north]{};
\filldraw (4,3.5) circle (3pt)node[anchor=south]{};
\filldraw (1.75,3.5) circle (3pt)node[anchor=east]{};

 \draw[black,thick]
(0,0) -- (3.5,0) -- (4,3.5) -- (0,0);

 \draw[black,thick]
(3.5,0) -- (4,3.5) -- (0,0);

 \draw[black,thick]
(3.5,0) -- (1.75,3.5) -- (0,0);

 \draw[black,thick]
(3.5,0) -- (-0.5,3.5) -- (0,0);

\draw[black,dashed]
(-1,3.5) -- (4.5,3.5);

        \end{tikzpicture}
          \hspace{0.5cm}
     \begin{tikzpicture}[scale=.4]
\filldraw (0,0) circle (3pt)node[anchor=east]{};
\filldraw (0,3.5) circle (3pt)node[anchor=east]{};
\filldraw (3.5,0) circle (3pt)node[anchor=north]{};
\filldraw (3.5,3.5) circle (3pt)node[anchor=south]{};

 \draw[black,thick]
(0,0) -- (0,3.5) -- (3.5,3.5) -- (3.5,0) -- (0,0);

\end{tikzpicture}
\end{center}
\vspace{-0.3cm}
\caption{Two frameworks in the Euclidean plane. The framework on the left is strongly regular but not completely regular as it has a colinear triple. The framework on the right is completely regular but not strongly regular as we can flatten the framework into colinear non-regular framework.}
\label{fig:notcomreg}
\end{figure}
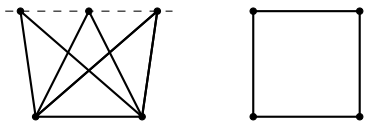

A framework $(G,p)$ in a normed space $X$ is: \emph{locally rigid} in $X$ if there exists a neighbourhood $U \subset X^V$ of $p$ where any equivalent framework in $U$ is congruent to $(G,p)$ (otherwise $(G,p)$ is locally flexible); \emph{globally rigid} if any equivalent frameworks $(G,q)$ in $X$ is congruent to $(G,p)$; \emph{infinitesimally rigid} in $X$ if $(G,p)$ is well-positioned and $\ker df_G(p) = \mathcal{T}(p)$ (otherwise $(G,p)$ is infinitesimally flexible);  \emph{independent} in $X$ if $(G,p)$ is well-positioned and $\rank df_G(p) = |E|$; 
	otherwise $G$ is dependent; \emph{minimally rigid} if it is infinitesimally rigid and $(G-e,p)$ is infinitesimally flexible for all $e\in E$.
	
It is easy to convince oneself that the framework on the left of Figure \ref{fig:notcomreg} is locally/infinitesimally rigid in the Euclidean plane, while the framework on the right is locally/infinitesimally flexible.
If we were to add a diagonal brace for the framework on the right, 
then we would obtain a locally/infinitesimally rigid framework for the Euclidean plane.
However such a framework will be locally/infinitesimally flexible with a non-Euclidean norm (Figure \ref{fig:k4-e}).

 \begin{figure}[htp]
\begin{tikzpicture}[scale=0.8]
\draw[thick] (0,0) -- (-0.794,0.794); 
\draw[thick] (0,0) -- (0.694,0.894); 

\draw[thick] (-0.794,0.794) -- (-0.1,1.687); 
\draw[thick] (-0.794,0.794) -- (0.694,0.894); 
\draw[thick] (0.694,0.894) -- (-0.1,1.687); 

\draw [dashed,domain=0:90] plot ({(cos(\x))^(2/3)}, {(sin(\x))^(2/3)});
\draw [dashed,domain=90:180] plot ({-(-cos(\x))^(2/3)}, {(sin(\x))^(2/3)});
\draw [dashed,domain=180:270] plot ({-(-cos(\x))^(2/3)}, {-(-sin(\x))^(2/3)});
\draw [dashed,domain=270:360] plot ({(cos(\x))^(2/3)}, {-(-sin(\x))^(2/3)});

\draw [dashed,domain=0:90] plot ({1.687*(cos(\x))^(2/3)}, {1.687*(sin(\x))^(2/3)});
\draw [dashed,domain=90:180] plot ({-1.687*(-cos(\x))^(2/3)}, {1.687*(sin(\x))^(2/3)});
\draw [dashed,domain=180:270] plot ({-1.687*(-cos(\x))^(2/3)}, {-1.687*(-sin(\x))^(2/3)});
\draw [dashed,domain=270:360] plot ({1.687*(cos(\x))^(2/3)}, {-1.687*(-sin(\x))^(2/3)});

\draw[fill] (0,0) circle [radius=0.05];
\draw[fill] (-0.794,0.794) circle [radius=0.05];
\draw[fill] (0.694,0.894) circle [radius=0.05];
\draw[fill] (-0.1,1.687) circle [radius=0.05];





\end{tikzpicture}\qquad\qquad
\begin{tikzpicture}[scale=0.8]


\draw [dashed,domain=0:90] plot ({(cos(\x))^(2/3)}, {(sin(\x))^(2/3)});
\draw [dashed,domain=90:180] plot ({-(-cos(\x))^(2/3)}, {(sin(\x))^(2/3)});
\draw [dashed,domain=180:270] plot ({-(-cos(\x))^(2/3)}, {-(-sin(\x))^(2/3)});
\draw [dashed,domain=270:360] plot ({(cos(\x))^(2/3)}, {-(-sin(\x))^(2/3)});

\draw [dashed,domain=0:90] plot ({1.687*(cos(\x))^(2/3)}, {1.687*(sin(\x))^(2/3)});
\draw [dashed,domain=90:180] plot ({-1.687*(-cos(\x))^(2/3)}, {1.687*(sin(\x))^(2/3)});
\draw [dashed,domain=180:270] plot ({-1.687*(-cos(\x))^(2/3)}, {-1.687*(-sin(\x))^(2/3)});
\draw [dashed,domain=270:360] plot ({1.687*(cos(\x))^(2/3)}, {-1.687*(-sin(\x))^(2/3)});

\draw[fill] (0,0) circle [radius=0.05];

\draw[thick] (0,0) -- (-1,-0.15); 
\draw[thick] (0,0) -- (0,-1); 

\draw[thick] (-1,-0.15) -- (0,-1); 
\draw[thick] (-1,-0.15) -- (-1,-1.15); 
\draw[thick] (0,-1) -- (-1,-1.15);

\draw[fill] (-1,-0.15) circle [radius=0.05];
\draw[fill] (0,-1) circle [radius=0.05];
\draw[fill] (-1,-1.15) circle [radius=0.05];

\draw[->,domain=100:190] plot ({2.2*cos(\x)}, {2.2*sin(\x)});

\end{tikzpicture}
\caption{(Left) An independent placement in $\ell_4^2$ of the complete graph on 4 vertices minus an edge. (Right) An equivalent independent framework that can be reached by continuous motion that preserves edge lengths. As the distance between the two non-adjacent vertices is altered during the motion, the framework is not locally rigid. The framework can be seen to be infinitesimally flexible by applying Theorem \ref{t:asimowroth}.}\label{fig:k4-e}
\end{figure}
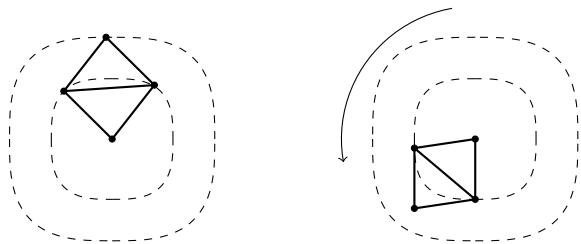

Global rigidity implies local rigidity,
however the converse is not always true (simply consider a tree when $X$ is 1-dimensional).
It is not too hard to show that a framework in $X$ is minimally rigid if and only if it is both infinitesimally rigid and independent (see, for example, \cite{dew2}).
We can also link infinitesimal and local rigidity.

\begin{thm}\label{t:asimowroth}\cite[Theorem 1.1]{D19}
	Let $(G,p)$ be a regular framework in a smooth normed space $X$.
	Then $(G,p)$ is infinitesimally rigid if and only if $(G,p)$ is locally rigid.
\end{thm}

\begin{thm}\cite[Theorem 10]{kitschulze}\label{t:kitschulze}
    Let finite $(G,p)$ be a well-positioned framework in a $d$-dimensional normed space $X$.
    Then the following hold:
    \begin{enumerate}[(i)]
        \item If $(G, p)$ is independent then $|E|= d|V| - \dim \ker df_G(p)$.
        \item If $(G, p)$ is infinitesimally rigid then $|E| \geq d|V|- \dim \mathcal{T}(p)$.
    \end{enumerate}
\end{thm}

We can improve Theorem \ref{t:kitschulze} in the special case when $d=2$. We say that $G=(V,E)$ is \emph{$(2,2)$-tight} if $|E|=2|V|-2$ and for every subgraph $(V',E')$ of $G$ we have $|E'|\leq 2|V'|-2$.

\begin{thm}\label{thm:rigidne}\cite[Theorem 1.4]{dew2}\label{thm:22normed}
	Let $X$ be a normed plane and $G=(V,E)$ be a graph.
	Then the following are equivalent.
	\begin{enumerate}[(i)]
	    \item There exists a minimally rigid framework $(G,p)$ in $X$.
	    \item $G$ is $(2,2)$-tight.
	\end{enumerate}
\end{thm}

Given a normed space $X$, a graph $G=(V,E)$ is: \emph{rigid} in $X$ if there exists a well-positioned placement of $G$ such that $(G,p)$ is infinitesimally rigid (otherwise $G$ is flexible); \emph{globally rigid} in $X$ if the set
	\begin{align*} 
	\grig (G;X) := \left\{ p \in X^V : (G,p) \text{ is globally rigid in } X \right\} 
	\end{align*}
	has a non-empty interior in $X^{V}$;
 \emph{independent} in $X$ if there exists a well-positioned placement of $G$ such that $(G,p)$ is independent;
 \emph{minimally rigid} in $X$ if there exists a well-positioned placement of $G$ such that $(G,p)$ is minimally rigid.
 In general normed spaces, rigidity and global rigidity can behave strangely, as can be seen in Figure \ref{fig:grexamples}.

  \begin{figure}[htp]
\begin{center}
     \begin{tikzpicture}[scale=0.8]
    \node[vertex] (a) at (2,0) {};
	\node[vertex] (b) at (1,-1) {};
    \node[vertex] (c) at (1,1) {};
    
	\node[vertex] (d) at (0,0) {};
	
    \node[vertex] (e) at (-2,0) {};
	\node[vertex] (f) at (-1,-1) {};
    \node[vertex] (g) at (-1,1) {};
	
	\draw[edge] (a)edge(b);
	\draw[edge] (a)edge(c);
	\draw[edge] (a)edge(d);
	\draw[edge] (b)edge(c);
	\draw[edge] (b)edge(d);
	\draw[edge] (c)edge(d);
    
	\draw[edge] (d)edge(e);
	\draw[edge] (d)edge(f);
	\draw[edge] (d)edge(g);
	\draw[edge] (e)edge(f);
	\draw[edge] (e)edge(g);
	\draw[edge] (f)edge(g);
	
	\filldraw[black] (0,-1.5) circle {};
\end{tikzpicture}\qquad \qquad
    \begin{tikzpicture}[scale=0.6]
    \node[vertex] (a) at (-2,-1) {};
	\node[vertex] (b) at (-2,1) {};
	
    \node[vertex] (c) at (0,1) {};
	\node[vertex] (d) at (0,-1) {};
	
    \node[vertex] (e) at (2,-1) {};
	\node[vertex] (f) at (2,1) {};

	\draw[edge] (a)edge(b);
	\draw[edge] (a)edge(c);
	\draw[edge] (a)edge(d);
	\draw[edge] (b)edge(c);
	\draw[edge] (b)edge(d);
	\draw[edge] (c)edge(d);
    
	\draw[edge] (c)edge(e);
	\draw[edge] (c)edge(f);
	\draw[edge] (d)edge(e);
	\draw[edge] (d)edge(f);
	\draw[edge] (e)edge(f);

\draw[thick,dashed]
(0,-2) -- (0,2);
\end{tikzpicture}
\end{center}
\vspace{-0.3cm}
\caption{(Left) A graph that is minimally rigid in any normed plane. Although the graph has a separating vertex, the normed plane itself does not have a continuous family of rotational isometries to exploit so as to flex the framework. The graph is also not globally rigid as we can translate the framework so that the separating vertex is at the origin and then map every vertex in the right copy of $K_4$ to its negative. (Right) The graph $H$ is globally rigid in all analytic normed planes (Theorem \ref{thm:H}). It is not globally rigid in the Euclidean plane; for almost all placements, the left two vertices may be reflected across the dashed line to obtain an equivalent but non-congruent framework.
However, the required reflection does not exist in most normed planes.}
\label{fig:grexamples}
\end{figure}
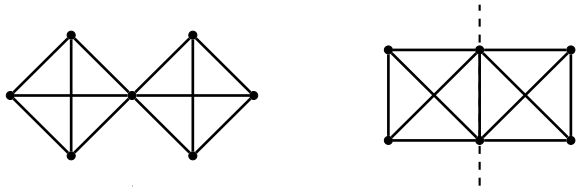

\section{Rigidity in normed spaces}
\label{sec:rigid}

\subsection{Analytic normed spaces}

A normed space $X$ is \emph{analytic} if the restriction of the norm to the set $X \setminus \{0\}$ is an analytic function.
An example of an analytic normed space is $\ell_{2n}^d$ for any positive integer $n$ (although in general $\ell_p^n$ is not analytic).
It can be shown that every norm can be approximated uniformly by analytic norms for the same linear space;
for example, see \cite[Theorem 2.5.2]{thompson}\footnote{The given reference describes approximating $d$-dimensional centrally symmetric convex bodies by those with analytic boundaries. However, the Minkowski functional of a convex body defines a homeomorphism between the space of $d$-dimensional centrally symmetric convex bodies with the topology gifted by the Hausdorff metric, and the space of norms for $\mathbb{R}^d$ with the compact-open topology. Furthermore, convex bodies with analytic boundaries will always have analytic Minkowski functionals.}.
It is immediate that every analytic norm is smooth.
They are also strictly convex.

\begin{lem}\label{prop:strconv}
    Every analytic normed space is strictly convex.
\end{lem}

\begin{proof}
    Suppose an analytic normed space $X$ is not strictly convex;
    i.e.~there exists $\|y_1\|=\|y_2\|=1$ with $\|t y_1 + (1-t)y_2\| = t\|y_1\| + (1-t)\|y_2\| = 1$ for all $t \in [0,1]$.
    Define $Y$ to be the linear span of $y_1 -y_2$ and define the analytic function $f : Y \rightarrow \mathbb{R}$ where for each $x = t (y_1 -y_2) \in Y$ (for some $t \in \mathbb{R}$) we define $f(x) := \|t (y_1-y_2) + y_2\| - 1$.
    As $f$ is an analytic function and $f(x) = 0$ for all $x= ty$ for $t \in [0,1]$, it is the constant zero function by Proposition \ref{p:anvar}.
    However this implies the unit ball contains the unbounded line through $y_1$ and $y_2$, contradicting compactness.
\end{proof}

As we shall see during this section, analytic normed spaces are well suited to rigidity theory as we can discuss ``generic'' properties without recourse to the usual polynomial notion of algebraic independence.

\begin{prop}\label{prop:regcreg}
	Let $X$ be an analytic normed space and $G=(V,E)$ be a graph.
	Then the sets $\reg (G;X)$ and $\creg (V;X)$ are open conull subsets of $X^{V}$.
\end{prop}

\begin{proof}
	If $\dim X=1$ then this holds immediately since any injective placement is completely regular.
	Suppose $\dim X \geq 2$ and choose any graph $G$.
	By \cite[Lemma 4.4]{D19},
	$\reg(G;X)$ is an open set.
	Choose any regular placement $p$ of $G$ with $\rank R(G,p) =n$.
	Let $A:= \{e_1, \ldots, e_n\} \subseteq E$ and $B := \{(v_1,m_1), \ldots, (v_n,m_n)\} \subseteq V \times \{1, \ldots,d\}$ be sets chosen so that the square submatrix $M(p)$ of $R(G,p)$ formed by the rows $A$ and columns $B$ has non-zero determinant.
	Define the function $\phi: X_G \rightarrow \mathbb{R}$ with $\phi(q) := \det M(q)$,
	where $M(q)$ is the square submatrix of $R(G,q)$ formed by the rows $A$ and columns $B$.
	As the norm $\| \cdot \|$ is analytic on $X \setminus \{0\}$,
	$\phi$ is an analytic function with an open connected domain.
	By Proposition \ref{p:anvar},
	the zero set $Z(\phi) := \{ q \in X_G : \phi(q) =0\}$ is a closed null set.
	It follows that $\reg(G;X) \supset X \setminus Z(\phi)$,
	hence $\reg(G;X)$ is conull.
	
	Let $\mathcal{G}$ be the set of graphs with vertex set $V$.
	As $\creg(V;X) = \bigcap_{H \in \mathcal{G}} \reg (H;X)$ then $\creg(V;X)$ is also an open conull set.
\end{proof}

To prove a similar result for the sets $\sreg(G;X)$ and $\csreg (V;X)$ we need Sard's theorem.

\begin{thm}(see \cite[Theorem 3.6.3]{manifold}\label{t:sard})
	Let $\mathcal{O} \subset \mathbb{R}^m$ be an open subset and $f: \mathcal{O} \rightarrow \mathbb{R}^n$ be $C^\infty$-differentiable.
	If we define
	\begin{align*}
		C := \left\{x \in \mathcal{O} : \rank df(x) < \sup_{y \in \mathcal{O}} \rank df(y) \right\}.
	\end{align*}
	then $f(C)$ is a null subset of $\mathbb{R}^n$.
\end{thm}


\begin{lem}\label{l:submer}
	Let $\mathcal{O} \subset \mathbb{R}^m$ be an open subset and $f: \mathcal{O} \rightarrow \mathbb{R}^n$ be a submersion\footnote{i.e.\ a differentiable map where the derivative is surjective at all points of the domain.}.
	If $A$ is a null subset of $\mathbb{R}^n$, then $f^{-1}(A)$ is a null subset of $\mathcal{O}$.
\end{lem}

\begin{proof}
	Choose any point $x \in \mathcal{O}$.
	By the constant rank theorem \cite[Theorem 2.5.15]{manifold},
	there exists an	open subset $U_x \subset \mathcal{O}$ with $x \in U_x$, open subsets $V_x,V' \subset \mathbb{R}^n$ with $f(U_x) \subset V_x$ and $0 \in V'$, an open subset $U' \subset \mathbb{R}^{m-n}$ with $0 \in U'$, and diffeomorphisms $\lambda: U_x \rightarrow U' \times V'$ and $\phi: V_x \rightarrow V'$ where $\lambda(x) = (0,0)$ and
	\begin{align*}
	    \phi \circ f \circ \lambda^{-1} : U' \times V' \rightarrow V', ~ (a,b) \mapsto b.
	\end{align*}
	As $\phi$ is a diffeomorphism, the set $\phi(A \cap V_x)$ is a null subset of $V'$ (see \cite[Lemma 3.6.1]{manifold}).
	Hence $U' \times \phi(A \cap V_x)$ is a null subset of $U' \times V'$.
	As
	\begin{align*}
	f^{-1}(A\cap V_x) \cap U_x = \lambda^{-1}\left ( \lambda \circ f^{-1} \circ \phi^{-1}[ \phi (A \cap V_x)]\right) = \lambda^{-1} \left( U' \times \phi (A \cap V_x) \right)
	\end{align*}
	and $\lambda$ is a diffeomorphism, the set $f^{-1}(A\cap V_x) \cap U_x$ is a null subset of $\mathcal{O}$ (see \cite[Lemma 3.6.1]{manifold}).
	By applying this same process to every point $x \in \O$,
	we obtain the open cover $\{U_x : x \in \mathcal{O}\}$ of $\O$ and the open cover $\{U_x : x \in \mathcal{O}\}$ of $f(\O)$,
	where for each $x \in \O$ we have $f(U_x) \subset V_x$ and the set $f^{-1}(A\cap V_x) \cap U_x$ is a null subset of $\mathcal{O}$.

	Choose a countable subset $\mathcal{O}' \subset \mathcal{O}$ so that $\{U_x : x \in \mathcal{O}'\}$ is an open cover of $\mathcal{O}$.
	Choose any $z \in f^{-1}(A)$ and pick a point $x \in \mathcal{O}'$ with $z \in U_x$.
	We note that $z \in f^{-1}(A \cap V_x) \cap U_x$,
	as $f^{-1}(V_x) \cap U_x =U_x$ and $z \in f^{-1}(A)$.
	It now follows that 
	\begin{align*}
	f^{-1}(A) \subset \bigcup_{x \in \mathcal{O}'} f^{-1}\left( A\cap V_x \right) \cap U_x.
	\end{align*}
	As each $f^{-1}(A \cap V_x) \cap U_x$ is a null subset of $\mathcal{O}$, the set $f^{-1}(A)$ is a null subset of $\mathcal{O}$.
\end{proof}

\begin{lem}\label{l:sard2}
	Let $X$ be a normed space, $G$ an independent graph in $X$, and $Z$ a null subset of $\mathbb{R}^{E}$.
	Suppose $\reg(G;X)$ is an open conull subset of $X^{V}$.
	Then the set $f_G^{-1}(Z)$ is a null subset of $X^V$.
\end{lem}

\begin{proof}
	The restriction of $f_G$ to $\reg(G;X)$ is a submersion,
	hence $f_G(\reg(G;X))$ is a non-empty open subset of $\mathbb{R}^{E}$ by the local onto theorem (see \cite[Theorem 3.5.2]{manifold}).
	We note that $f_G(\reg(G;X)) \cap Z$ is a null set thus, by Lemma \ref{l:submer}, the set of all regular placements $p$ with $f_G(p) \in Z$ is also a null set.
	As $X^V \setminus \reg(G;X)$ is a null set, the result holds.
\end{proof}

\begin{prop}\label{p:comreg}
	Let $X$ be an analytic normed space and $G=(V,E)$ be a graph.
	Then the sets $\sreg (G;X)$ and $\csreg (V;X)$ are open conull sets.
\end{prop}

\begin{proof}
	For each of the $k$ connected components of $G$, fix a vertex $w_i$.
	Define $Y := \{ (x_v)_{v \in V} : x_{w_i}=0 \text{ for all } 1 \leq i \leq k \}$ and $f := f_G|_Y$.
	Since $f$ is closed it is proper and so the set $C := f(X^V \setminus \reg(G;X))$ is closed.
	Also the set $\sreg(G;X) = X^V \setminus  f_G^{-1}(C)$ is open.
	
	If $G$ is independent then Theorem \ref{t:sard} and Lemma \ref{l:sard2} imply that
	$\sreg(G;X)$ is a conull subset of $X^V$.
	So $G$ is not independent.
	Let $S$ be the set of maximal independent spanning subgraphs of $G$.
	We note that the set $\bigcup_{H \in S} \sreg(H;X)$ is conull.
	Choose any placement $p$ from $\bigcup_{H \in S} \sreg(H;X)$ and any $p' \in f_G^{-1}[f_G(p)]$;
	we may suppose $p \in \sreg(H;X)$ for some $H \in S$.
	As $p' \in f_H^{-1}[f_H(p)]$ and $H$ is a maximally independent spanning subgraph of $G$,
	we have
	\begin{align*}
		\rank df_G(p') \geq \rank df_H(p') = \rank df_H(p) = \rank df_G(p),
	\end{align*}
	thus $p \in \sreg(G;X)$. 
	It now follows that 
	\begin{align*}
		\bigcup_{H \in S} \sreg(H;X) \subset \sreg(G;X), 
	\end{align*}
	and thus $\sreg(G;X)$ is a conull subset of $X^V$.
	Let $\mathcal{G}$ be the set of graphs with vertex set $V$.
	As $\csreg(V;X) = \bigcap_{H \in \mathcal{G}} \sreg (H;X)$ then $\csreg(V;X)$ is an open conull subset of $X^V$ also.
\end{proof}

\subsection{Global rigidity in smooth normed spaces with finitely many linear isometries}
At this point, to prove that a graph is globally rigid requires proving that every placement in an open set is globally rigid.
In this subsection we shall prove that it suffices to find a single placement that is both globally rigid and infinitesimally rigid.

First we need a definition.
Let $X$ and $Y$ be normed spaces, $U \subset X$ be an open subset and $f :U \rightarrow Y$ be a continuous map.
We say $f$ is \emph{proper} if $f^{-1}[C]$ is compact for all compact sets $C$;
equivalently,
$f$ is proper if every sequence $(x_n)_{n \in \mathbb{N}}$ in $U$, where $(f(x_n))_{n \in \mathbb{N}}$ is bounded, is itself bounded.

\begin{lem}\label{l:infglobal1}
	Let $X$ and $Y$ be normed spaces, $U \subset X$ be an open subset and $f :U \rightarrow Y$ be a proper $C^1$-differentiable map.
	Suppose for some point $x \in U$ there exists $k \in \mathbb{N}$ such that $|f^{-1} f[x]| = k$,
	and all points of $f^{-1} f[x]$ are regular.
	Then there exists an open neighbourhood $U' \subset U$ of $x$ where $|f^{-1} f[x']| \leq k$ for all $x' \in U'$.	
\end{lem}

\begin{proof}
	Suppose there exists a sequence $(x_n)_{n\in \mathbb{N}}$ in $U$ where $x_n \rightarrow x$ as $n \rightarrow \infty$ and $f^{-1}f[x_n] > k$.
	By assumption we may choose $k+1$ sequences $(y^1_{n})_{n\in \mathbb{N}}, \ldots, (y^{k+1}_n)_{n\in \mathbb{N}}$ where $y^1_n=x_n$,
	$y^i_n \neq y^j_n$ and $f(y^j_n)= f(x_n)$ for each $n \in \mathbb{N}$ and $1\leq i < j \leq k+1$.
	As $f$ is proper,
	each sequence $(y^i_n)_{n\in \mathbb{N}}$ is bounded.
	By the Bolzano-Weierstrass Theorem,
	we may assume that we chose $(x_n)_{n\in \mathbb{N}}$ such that each sequence $(y^{i}_n)_{n\in \mathbb{N}}$ converges to a limit $y^i$.
	As $f(x) = f(y^i)$ for each $i$ then $y^1, \ldots,y^{k+1} \in f^{-1} f[x]$.
	Since $|f^{-1} f[x]| = k$,
	there exists $i \neq j$ such that $y^i = y^j$.
	
	Let $x' := y^i = y^j$.
	Since $x'$ is regular and $|f^{-1} f[x']| = k$, it follows from the constant rank theorem (see \cite[Theorem 2.5.15]{manifold}) that there exists an open neighbourhood $U'$ of $x'$ where $f|_{U'}$ is injective.
	Since both $(y^{i}_n)_{n\in \mathbb{N}}$ and $(y^{j}_n)_{n\in \mathbb{N}}$ converge to $x'$,
	we have $y^i_n=y^j_n$ for sufficiently large $n$. This contradicts our hypotheses.
\end{proof}

\begin{lem}\label{l:infglobal2}
	Let $X$ be a normed space with $k$ linear isometries and $(G,p)$ be isometrically full.
	Suppose $p_{v_0} =0$ for some $v_0 \in V$.
	Then given
	\begin{align*}
		S:= \{ p' \in X^V : f_G(p')= f_G(p), ~ p'_{v_0} = 0 \},
	\end{align*} 
	$(G,p)$ is globally rigid if and only if $|S| \leq k$.
\end{lem}

\begin{proof}
	We first note that $\Isolin (X)\cdot p \subseteq S$,
	and $(G,p)$ is globally rigid if and only if $S = \Isolin (X)\cdot p$.
	As $p$ is isometrically full,
	we have that $|\Isolin(X) \cdot p|=k$.
	Hence $|S|\leq k$ if and only if $|S|=k$ if and only if $S = \Isolin (X)\cdot p$.
\end{proof}

\begin{lem}\label{l:quotcount}
	Let $X$ be a $d$-dimensional smooth normed space,
	$G = (V,E)$ and $p \in X_G$ be a placement of $G$ with $\dim \ker df_G(p) < d-1 + |V|$.
	Then $p$ is isometrically full.
\end{lem}

\begin{proof}
	Suppose for contradiction that $\dim \ker df_G(p) < d-1 + |V|$ but $p$ is not isometrically full.
	By translating if necessary,
	we will assume $p_{v_0}=0$ for some $v_0 \in V$.
	Define $Y := \spann \{ p_v : v \in V\}$ and $W := \spann \{\varphi_y : y \in Y\}$.
	Since $p$ is not isometrically full,
	there exists a non-trivial linear isometry $g$ where $g(y)=y$ for all $y \in Y$.
	Denote the dual space\footnote{This is the normed space of all linear functions $f:X \rightarrow \mathbb{R}$ on $X$ with the norm $\|f\|^* := \sup_{\|x\|=1}|f(x)|$.} of $X$ by $X^*$.
	It is immediate that $W$ is a subspace of $X^*$.
	First suppose that $W=X^*$.
	Define the map $T:X^* \rightarrow X^*$ where $T(f) = f \circ g^{-1}$ for all $f \in X^*$.
	Since $g$ is an isometry of $X$,
	the map $T$ is a linear isometry of the normed space $X^*$.
	We recall that the support functional $\varphi_x$ is the derivative of $\frac{1}{2}\|\cdot\|^2$ at a point $x$.
	Hence, for any $x,y \in X$,
	\begin{align*}
		T(\varphi_y)(x) = \lim_{t \rightarrow 0} \frac{1}{t}\left( \left\| y + t g^{-1}(x) \right\| - \|y\|\right) = 
		\lim_{t \rightarrow 0} \frac{1}{t}\left( \left\| g(y) + t (x) \right\| - \|g(y)\|\right) = \varphi_{g(y)}(x).
	\end{align*}
	As $g(y) = y$, we have $T(\varphi_y) = \varphi_{y}$ for any $y \in Y$.
	Since $T$ is linear we have $T(f)=f$ for all $f \in X^*$.
	However, this implies $g$ is the identity map,
	a contradiction.
	Now suppose that $W \neq X^*$,
	i.e.~there exists non-zero $x \in X$ where $f(x)=0$ for all $f \in W$.
	Define for each $w \in V$ the vector $u^w =( u^w_v)_{v \in V} \in X^V$,
	where for each $v \in V$ we have
	\begin{align*}
		u^w_v :=
		\begin{cases}
			x, & \text{if } v =w, \\
			0, & \text{otherwise.}
		\end{cases}
	\end{align*}
	Then $df_G(p)(u^w) =0$ for each $w \in V$.
	If we choose vectors $B := \{x_1, \ldots,x_{n-1} \} \subset X$ such that $B \cup \{x\}$ is a basis of $X$,
	then
	\begin{align*}
	    S := \left\{(x_i)_{v \in V} :1 \leq i \leq d-1 \right\} \cup \left\{u^w :w \in V \right\}
	\end{align*}
	is a linearly independent subset of $\ker df_G(p)$ with $|S| = d-1 +|V|$.
	However this implies $\dim \ker df_G(p) \geq d-1 +|V|$ which
	contradicts our hypothesis.
\end{proof}

We can now prove our first main result, an analogue of the ``averaging theorem'' of Connelly and Whiteley \cite{CW}.

\begin{thm}\label{thm:ave}
	Let $(G,p)$ be globally rigid and infinitesimally rigid in a smooth normed space $X$ with a finite number of linear isometries.
	Then there exists an open neighbourhood $U$ of $p$ such that $(G,q)$ is globally rigid and infinitesimally rigid in $X$ for all $q \in U$.
\end{thm}

\begin{proof}
	By translating $(G,p)$ we may assume $p_{v_0} = 0$ for some $v_0 \in V$.
Define 
	\begin{align*}
		Y := \{ x \in X^V : x_{v_0} = 0 \}, \quad Z := \mathbb{R}^{E}, \quad f:= f_G|_Y, \quad k := |f^{-1}(p)|.
	\end{align*}
	The graph $G$ must be connected as $(G,p)$ is infinitesimally rigid,
	so the map $f:Y \rightarrow Z$ is proper.
	By Lemma \ref{l:infglobal1},
	there exists an open neighbourhood $U \subset Y$ of $p$ where $\left|f^{-1} f[q] \right| \leq k$ for all $q \in U$.
	Since the set of regular placements of a graph in $X$ is open,
	we may assume we chose $U$ sufficiently small such that each $q \in U$ will be infinitesimally rigid.
    As $\Isolin(X)$ is a finite set,
	$\dim\mathcal{T}(q) = d$ for every $q \in U$.
	By Theorem \ref{t:kitschulze} and Lemma \ref{l:quotcount},
	each $q \in U$ is isometrically full.
	We can now apply Lemma \ref{l:infglobal2} to any $q \in U$ to see that the framework $(G,q)$ is globally rigid.
\end{proof}

We can immediately state the following corollary.

\begin{cor}\label{cor:ave}
    Let $X$ be an analytic normed space with finitely many linear isometries.
	Then a graph $G$ is globally rigid in $X$ if and only if there exists an infinitesimally and globally rigid framework $(G,p)$ in $X$.
\end{cor}

\begin{proof}
    If there exists an infinitesimally and globally rigid framework $(G,p)$ in $X$ then $G$ is globally rigid in $X$ by Theorem \ref{thm:ave}.
    Now suppose $G$ is globally rigid in $X$.
    By Proposition \ref{prop:regcreg},
    there exists a regular framework $(G,p)$ in $X$ that is globally rigid.
    Any globally rigid framework is locally rigid,
    hence $(G,p)$ is infinitesimally rigid by Theorem \ref{t:asimowroth}.
\end{proof}

\section{Generalising rigid body motions}
\label{sec:reflect}

\subsection{Reflections and rotations in analytic normed planes}

We first need the following useful results.

\begin{lem}\label{l:2point}\cite[Lemma 4.18]{dew2}
	Let $X$ be a strictly convex normed plane and $x,y \in X$ be distinct points.
	For any positive real numbers $r_x,r_y$, we have three possibilities:
	\begin{enumerate}[(i)]
		\item If $r_x +r_y < \|x-y\|$ or $|r_x -r_y| > \|x-y\|$ then there is no point $z \in X$ where $\|z-x\|=r_x$ and $\|z-y\|=r_y$.
		\item If $r_x + r_y = \|x-y\|$ or $|r_x - r_y|=\|x-y\|$,
		then there exists a unique point $z \in X$ where $\|z-x\|=r_x$ and $\|z-y\|=r_y$.
		Furthermore,
		$z$ lies on the line through $x,y$.
		\item If $|r_x - r_y|< \|x-y\| <r_x + r_y$,
		then there exists a unique pair of distinct points $z_1, z_2 \in X$ where $\|z_i-x\|=r_x$ and $\|z_i-y\|=r_y$ for $i \in \{1,2\}$.
		Furthermore,
		the line through $x,y$ separates the points $z_1,z_2$.
	\end{enumerate}
\end{lem}

\begin{thm}\cite{Brouwer12}\label{thm:iod}
    Let $U \subset \mathbb{R}^m$ be an open set and $f : U \rightarrow \mathbb{R}^n$ be a continuous injective map.
    Then $f(U)$ is open and $f|^{f(U)}$ is a homeomorphism.
\end{thm}

Suppose $R:\mathbb{R}^2 \rightarrow \mathbb{R}^2$ is a (Euclidean) reflection about the line $\ell$ through the origin.
If we choose a point $z \in \ell \setminus \{0\}$ then we notice that $\| R(x) \| = \| x \|$ and $\| R(x) - z\| = \| x - z\|$ for all $x \in \mathbb{R}^2$,
and $R(x)=x$ if and only if $x \in \ell$.
In fact as the norm is Euclidean, $R$ is the unique map to have this property;
this follows from Lemma \ref{l:2point}. 
Hence we can define the reflection with respect to the two points $0$ and $z$.
We will now show that this definition is amenable to extensions to other normed planes.

\begin{thm}\label{t:genrefl}
	Let $X$ be an analytic normed plane	and choose any non-zero point $z \in X$.
	Then there exists a unique map $R_z :X \rightarrow X$ with the following properties:
	\begin{enumerate}[(i)]
		\item
		If $x$ does not lie on the line through $0,z$,
		then $R_z(x)$ is the unique point distinct from $x$ where $\| R_z(x) \| = \| x \|$ and $\| R_z(x) - z\| = \| x - z\|$.
		\item
		If $x$ does lie on the line through $0,z$,
		then $R_z(x) = x$.
	\end{enumerate}
	Furthermore, the map $R_z$ will have the following properties:
	\begin{enumerate}
		\item \label{l:generalreflection1} $R_z$ is a homeomorphism.
		\item \label{l:generalreflection2} $R_z$ is an analytic immersion on the set of points that do not lie on the line through $0,z$.
		\item \label{l:generalreflection3} If $x \notin \spann\{z\}$ then $R_z(x)$ will not lie in the same connected component of $\mathbb{R}^2 \setminus \spann \{z\}$ as $x$.
		\item \label{l:generalreflection4} If $x \in \spann\{z\}$ and $y$ is another point where $\| y \| = \| x \|$ and $\| y - z\| = \| x -z\|$, then $y=x$.
	\end{enumerate}
\end{thm}

\begin{proof}
    It follows from Lemma \ref{l:2point} that there exists a unique map $R_z$ as described;
    furthermore, such a map must be bijective and satisfy both (\ref{l:generalreflection3}) and (\ref{l:generalreflection4}).
    We now prove  (\ref{l:generalreflection1}) and (\ref{l:generalreflection2}).
    For the former it will be sufficient to prove that $R_z$ is continuous by Theorem \ref{thm:iod}.
    Define the map
	\begin{align*}
		f : X \rightarrow \mathbb{R}^2, ~ x \mapsto \left(\frac{1}{2}\| x\|^2,\frac{1}{2}\| x - z\|^2 \right).
	\end{align*}
	Then $f$ is analytic,
	and given a basis $b_1,b_2$ of $X$ and $x \in X$,
	we can represent $df(x)$ by the $2 \times 2$ matrix
	\begin{align*}
		\begin{bmatrix}
			\varphi_x(b_1) & \varphi_x(b_2) \\
			\varphi_{x-z}(b_1) & \varphi_{x-z}(b_2)
		\end{bmatrix}.
	\end{align*}
	Since $X$ is strictly convex (Proposition \ref{prop:strconv}) it follows that $\rank df(x) =2$ if and only if $x$ does not lie on the line through $0,z$.
	Define $\H,\H'$ to be the two open half-planes of $X$ formed by the line through $0,z$.
	As $\H,\H'$ are locally compact and connected,
	Theorem \ref{thm:imt} implies that the restricted maps $f|^{f(\H)}_\H$ and $f|^{f(\H')}_{\H'}$ are analytic local diffeomorphisms.
	By Lemma \ref{l:2point},
	both $f|^{f(\H)}_\H$ and $f|^{f(\H')}_{\H'}$ are injective,
	hence both are analytic diffeomorphisms.
	
	By Lemma \ref{l:2point}, we have $f(\H)=f(\H')$.
	By the uniqueness of $R_z$ we have, for each $x \in X$,
	\begin{align*}
		R_z (x)=
		\begin{cases}
			(f|^{f(\H')}_{\H'})^{-1} \circ f|^{f(\H)}_\H(x), & \text{if } x \in \H, \\
			(f|^{f(\H)}_\H)^{-1} \circ f|^{f(\H')}_{\H'}(x), & \text{if } x \in \H', \\
			x, & \text{if } x \in \spann\{z\}.
		\end{cases}
	\end{align*}
	Hence $R_z$ is analytic on $\H \cup \H'$ with $\rank d R_z (x) = 2$ for all $x \in \H \cup \H'$;
	i.e. (\ref{l:generalreflection2}) holds.
	
	Choose any point $y$ on the line through $0,z$ and let $(y_n)_{n \in \mathbb{N}}$ be any sequence that converges to $y$.
	As $f$ is continuous and $f(R_z(y_n)) = f(y_n) \rightarrow f(y)$ as $n \rightarrow \infty$,
	the sequence $(R_z(y_n))_{n \in \mathbb{N}}$ converges to a point in $f^{-1}[f(y)]$.
	Since $y$ lies on the line through $0,z$,
	it follows that $\lim_{n \rightarrow \infty} R_z(y_n) = y = R_z(y)$,
	hence $R_z$ is continuous at $y$.
	(\ref{l:generalreflection1}) now holds because $R_z$ is continuous.
\end{proof}

We can also define rotations to be the unique maps that preserve the distance to two points while one rotates around the other,
with some added continuity assumptions to stop points being ``reflected'' in a certain sense.

\begin{thm}\label{t:genrot}
	Let $X$ be an analytic normed plane	and choose any non-zero point $z \in X$,
	and let $f: \mathbb{T} \rightarrow S_{\|z\|}[0]$ be an analytic bijective map with $f(0)=z$ and $\|f'(t)\|=1$ for all $t \in \mathbb{T}$.
	Then there exists a unique continuous map $r_z :X \times \mathbb{T} \rightarrow X$ with the following properties:
	\begin{enumerate}[(i)]
		\item \label{l:genrot0} $r_z(z,t) = f(t)$ for all $t \in \mathbb{T}$.
		\item \label{l:genrot} $r_z(x,0) = x$ for all $x \in X$.
		\item \label{l:genrot3} $\| r_z(x,t) \| = \| x \|$ and $\| r_z(x,t) - r_z(z,t)\| = \| x - z\|$ for all $x \in X$ and $t \in \mathbb{T}$.
	\end{enumerate}
	Furthermore, the map $r_z$ will have the following properties:
	\begin{enumerate}
		\item \label{l:genrot2} $r_z$ is analytic on the set $(X \setminus \spann\{z\}) \times \mathbb{T}$.
		\item \label{l:genrot4} For any $x,y \in X \setminus \spann\{z\}$ and any $t \in \mathbb{T}$,
		$x,y$ lie in the same connected component of $X \setminus \spann\{z\}$ if and only if $r_z(x,t),r_z(y,t)$ lie in the same connected component of $X \setminus \spann\{f(t)\}$.
		\item \label{l:genrot1.5} For any $t \in \mathbb{T}$,
		$r_z(\cdot,t)$ is a homeomorphism.
		\item \label{l:genrot1.6} For any non-zero $x \in X$,
		$r_z(x,\cdot)$ is a homeomorphism between $\mathbb{T}$ and $S_{\|x\|}[0]$.
		\item \label{l:genrot5} For any $x \in X$ and any $t \in \mathbb{T}$,
		$r_z(x,-t) = -r_z(x,t)$.
	\end{enumerate}
\end{thm}

\begin{proof}
	For each $t \in \mathbb{T}$ define the open half-planes $\H^+_t,\H^-_t$ of $X$ formed by the line through $0,z$,
	where $\H^+_t$ contains all the points $f(t +h)$ for $h \in (0,\pi)$ and $\H^-_t$ contains all the points $f(t -h)$ for $h \in (0,\pi)$.
	With this we define the sets $\mathcal{H}^+ := \{ (y,t) : y \in \H^+_t\}$ and $\mathcal{H}^- := \{ (y,t) : y \in \H^-_t\}$,
	and we note both sets are open and path-connected.
	
	We will now prove that a map with the properties of $r_z$ exists;
	the uniqueness of such a map will then follow automatically from Lemma \ref{l:2point}.
	Define the analytic map
	\begin{align*}
		g : X \times \mathbb{T} \rightarrow \mathbb{R}^2, ~ (x,t) \mapsto \left( \frac{1}{2}\|x\|^2, \frac{1}{2}\|x-f(t)\|^2 \right).
	\end{align*}
	Fix a basis $b_1,b_2$ of $X$.
	For any $x \in X$,
	we can represent $dg(x)$ by the $2 \times 3$ matrix
	\begin{align*}
		\begin{bmatrix}
			\varphi_x(b_1) & \varphi_x(b_2) & 0 \\
			\varphi_{x-f(t)}(b_1) & \varphi_{x-f(t)}(b_2) & -\varphi_{x-f(t)}(f'(t))
		\end{bmatrix}.
	\end{align*}
	Since $X$ is strictly convex (Proposition \ref{prop:strconv}), any principal $2 \times 2$ submatrix of the above matrix is non-singular if and only if $\rank dg(x,t) =2$ if and only if $x$ does not lie on the line through $0,f(t)$.
	As $g$ is analytic and $\rank dg(x,t) =2$ at every point $(x,t)$ of $\mathcal{H}^{+/-} :=\mathcal{H}^+ \cup \mathcal{H}^-$ then $g(\mathcal{H}^{+/-})$, $g(\mathcal{H}^+)$ and $g(\mathcal{H}^-)$ are open sets, with the latter two also being connected.
    Hence Theorem \ref{thm:ifta} implies that there exists a unique analytic map $h :g(\mathcal{H}^{+/-}) \times \mathbb{T} \rightarrow X$ where
	\begin{align*}
		g( h(g(x,t),t'),t')= g(x,t)
	\end{align*}
	for any $(x,t) \in \mathcal{H}^{+/-}$ and $t' \in \mathbb{T}$.
	Furthermore,
	if $(x,t) \in \mathcal{H}^+$ (respectively, $\mathcal{H}^-$) then $h((x,t),t') \in \mathcal{H}^+$ (respectively, $\mathcal{H}^-$) for all $t' \in \mathbb{T}$.
	We now define the map $r_z :X \times \mathbb{T} \rightarrow X$,
	where for any $(x,t) \in X \times \mathbb{T}$ we have
	\begin{align*}
		r_z :X \times \mathbb{T} \rightarrow X, ~ r_z(x,t) := 
		\begin{cases}
			h(g(x,0),t) &\text{ if } x \in X \setminus \spann \{ z\}, \\
			\alpha f(t), &\text{ if } x = \alpha z.
		\end{cases}
	\end{align*}
	The map $r_z$ is analytic (and hence continuous) on the set $(X \setminus \spann \{ z\}) \times \mathbb{T}$.
	To see that $r_z$ is continuous at every point,
	choose a convergent sequence $((x_n,t_n))_{n \in \mathbb{N}}$ with limit $(\alpha z,t)$.
	Since $g$ is continuous and $\lim_{n \rightarrow \infty} g(r_z(x_n,t_n)) = \lim_{n \rightarrow \infty} g(x_n,t_n) =g(\alpha z,t)$,
    if follows that $\lim_{n \rightarrow \infty} r_z(x_n,t_n) = \alpha f(t)$ by Lemma \ref{l:2point}.
	
	We will now prove  (\ref{l:genrot2})--(\ref{l:genrot5}).
	As stated earlier,
	$r_z$ is analytic on the set $(X \setminus \spann \{ z\}) \times \mathbb{T}$,
	hence the map satisfies (\ref{l:genrot2}).
	Choose any $x \in \H^+_0$.
	By the continuity of $r_z$ and Lemma \ref{l:2point},
	we must have $r_z(x,t) \in \H^+_t$ for any $t \in \mathbb{T}$ (as otherwise there would exists $r_z(x,t) \in \spann \{ f(t)\}$ which would imply $x \in \spann \{z\}$).
	As a similar result can be obtained for any point $x \in \H^-_0$,
	(\ref{l:genrot4}) also holds.
	Choose any $t \in \mathbb{T}$ and points $x,y \in X$.
	If $r_z(x,t)=r_z(y,t)$ then, by Lemma \ref{l:2point}, we have $x=y$,
	hence $r_z(\cdot,t)$ is injective.
	Now choose $w \in X$ with $c_0 := \|w\|$ and $c_z:= \|w- r_z(z,t)\|$.
	By Lemma \ref{l:2point} there exists a point $u \in X$ where $r_z(u,t) = w$,
	hence $r_z(\cdot,t)$ is surjective.
	As $r_z(\cdot,t)$ is a continuous bijective map,
	(\ref{l:genrot1.5}) holds by Theorem \ref{thm:iod}.
	
	Choose any $x \in X$ and any $s,t \in \mathbb{T}$.
	If $r_z(x,s)=r_z(x,t)$ then, by Lemma \ref{l:2point}, we have $f(s)=f(t)$.
	As $f$ is a diffeomorphism,
	it follows that $r_z(x,\cdot)$ is injective.
	The map $r_z(x,\cdot)$ is closed as it is continuous with compact domain and Hausdorff codomain,
	and open by Theorem \ref{thm:iod} as it is injective.
	Hence $r_z(x,\cdot)$ is a continuous open and closed bijective map between $\mathbb{T}$ and $S_{\|x\|}[0]$, as $r_z(x, \mathbb{T})$ is a clopen subset of the connected set $S_{\|x\|}[0]$.
	(\ref{l:genrot1.6}) now holds as all continuous open bijections are homeomorphisms.
	As $\|f'(t)\|=1$ for all $t \in \mathbb{T}$,
	it follows that $f(-t) = -f(t)$.
	Hence (\ref{l:genrot5}) follows from the continuity of $r_z$ and Lemma \ref{l:2point}.
\end{proof}

\begin{rem}
    For any given closed analytic path $f$ around $S_{\|z\|}[0]$ with $f(0)=z$ and $\|f'(t)\|=1$ for all $t \in \mathbb{T}$, the only other possible path with such properties is $t \mapsto f(2\pi-t)$.
    Hence we shall refer to the map $r_z$ without mentioning the chosen analytic path $f$.
\end{rem}

Using Theorems \ref{t:genrot} and \ref{t:genrefl} for some non-zero $z \in X$,
we define the following:
\begin{enumerate}[(i)]
	\item The map $R_z$ is called the \emph{$z$-reflection}.
	\item The map $r_z$ is called the \emph{$z$-rotation}.
	\item The map $x \mapsto r_z(x,\theta)$ is called the \emph{$z$-rotation by $\theta$}.
	\item If a map $f:X \rightarrow X$ has the property that $f=R_z$ for some non-zero $z \in X$ then $f$ is a \emph{reflection},
	and if $f = r_z (\cdot ,\theta)$ for some non-zero $z \in X$ then $f$ is a \emph{rotation by $\theta$}.
	\item If a map $g:X \times \mathbb{T} \rightarrow X$ has the property that $g=r_z$ for some non-zero $z \in X$ then $g$ is a \emph{rotation}.
\end{enumerate}
See Figure \ref{fig:rotrefl} for examples of $z$-reflections and $z$-rotations for $\ell_4^2$.

\begin{figure}[htp]
\begin{tikzpicture}[scale=1.2]

\draw [dashed,domain=0:90] plot ({(cos(\x))^(2/3)}, {(sin(\x))^(2/3)});
\draw [dashed,domain=90:180] plot ({-(-cos(\x))^(2/3)}, {(sin(\x))^(2/3)});
\draw [dashed,domain=180:270] plot ({-(-cos(\x))^(2/3)}, {-(-sin(\x))^(2/3)});
\draw [dashed,domain=270:360] plot ({(cos(\x))^(2/3)}, {-(-sin(\x))^(2/3)});

\draw [dashed,domain=0:90] plot ({1.3*(cos(\x))^(2/3) +1.6}, {1.3*(sin(\x))^(2/3)+0.7});
\draw [dashed,domain=90:180] plot ({-1.3*(-cos(\x))^(2/3) +1.6}, {1.3*(sin(\x))^(2/3)+0.7});
\draw [dashed,domain=180:270] plot ({-1.3*(-cos(\x))^(2/3) +1.6}, {-1.3*(-sin(\x))^(2/3)+0.7});
\draw [dashed,domain=270:360] plot ({1.3*(cos(\x))^(2/3) +1.6}, {-1.3*(-sin(\x))^(2/3)+0.7});

\node[vertex] (o) at (0,0) {};
\node[vertex] (z) at (1.6,0.7) {};
\node[vertex] (x) at (0.32,0.97) {};
\node[fadedvertex] (y) at (0.95,-0.55) {};

\node (olabel) at (-0.3,-0.3) {\tiny$(0,0)$};
\node (zlabel) at (1.9,0.7) {\small$z$};
\node (xlabel) at (0.02,1.27) {\small$x$};
\node (ylabel) at (1.35,-1) {\small$R_z(x)$};

\draw[edge] (o)edge(z) {};
\draw[edge] (o)edge(x) {};
\draw[edge] (x)edge(z) {};
\draw[fadededge] (o)edge(y) {};
\draw[fadededge] (y)edge(z) {};





\end{tikzpicture}\qquad\qquad
\begin{tikzpicture}[scale=1.2]

\draw [dashed,domain=0:90] plot ({(cos(\x))^(2/3)}, {(sin(\x))^(2/3)});
\draw [dashed,domain=90:180] plot ({-(-cos(\x))^(2/3)}, {(sin(\x))^(2/3)});
\draw [dashed,domain=180:270] plot ({-(-cos(\x))^(2/3)}, {-(-sin(\x))^(2/3)});
\draw [dashed,domain=270:360] plot ({(cos(\x))^(2/3)}, {-(-sin(\x))^(2/3)});

\draw [dashed,domain=0:90] plot ({1.687*(cos(\x))^(2/3)}, {1.687*(sin(\x))^(2/3)});
\draw [dashed,domain=90:180] plot ({-1.687*(-cos(\x))^(2/3)}, {1.687*(sin(\x))^(2/3)});
\draw [dashed,domain=180:270] plot ({-1.687*(-cos(\x))^(2/3)}, {-1.687*(-sin(\x))^(2/3)});
\draw [dashed,domain=270:360] plot ({1.687*(cos(\x))^(2/3)}, {-1.687*(-sin(\x))^(2/3)});

\node[vertex] (o) at (0,0) {};
\node[vertex] (z) at (0,1.687) {};
\node[vertex] (x) at (-0.794,0.794) {};

\draw[edge] (o)edge(z) {};
\draw[edge] (o)edge(x) {};
\draw[edge] (x)edge(z) {};

\node (olabel) at (0.5,0) {\tiny$(0,0)$};
\node (zlabel) at (0, 2) {\small$z$};
\node (xlabel) at (-1.2,0.794) {\small$x$};

\draw[->,domain=100:190] plot ({2.2*cos(\x)}, {2.2*sin(\x)});

\node[fadedvertex] (z') at (-1.34,-1.34) {};
\node[fadedvertex] (x') at (-0.3,-0.995) {};

\node (z'label) at (-1.8,-1.7) {\small$r_z(z,t)$};
\node (x'label) at (-0.1,-1.3) {\small$r_z(x,t)$};

\draw[fadededge] (o)edge(z') {};
\draw[fadededge] (o)edge(x') {};
\draw[fadededge] (x')edge(z') {};

\draw[->,domain=100:190] plot ({2.2*cos(\x)}, {2.2*sin(\x)});

\end{tikzpicture}
\caption{(Left) A $z$-reflection in $\ell_4^2$ of a point $x$; the grey point represents its new position $R_z(x)$. (Right) A $z$-rotation in $\ell_4^2$ of a point $x$; the grey triangle represents the new positions $r_z(x,t)$ and $r_z(z,t)$ of $x$ and $z$ respectively.}\label{fig:rotrefl}
\end{figure}
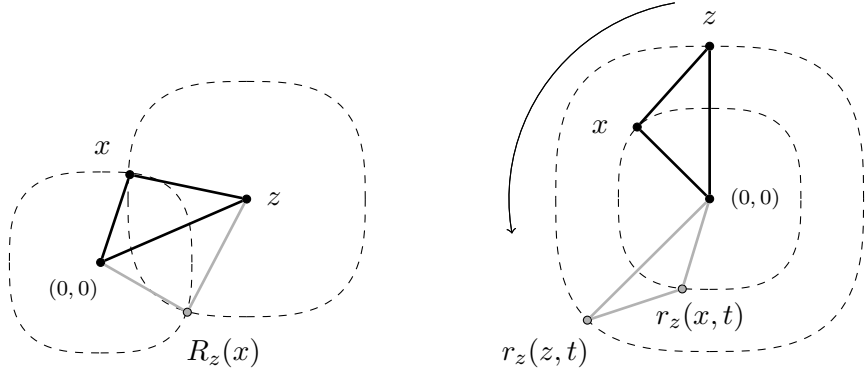

The following result proves that we can at times change the order of rotations and reflections.

\begin{cor}\label{c:switch}
    Let $X$ be an analytic normed plane and $z \in X$ a non-zero point.
    Then for any $x \in X$ and $t \in \mathbb{T}$ we have
    \begin{align*}
        r_z( R_z(x),t) = R_{r_z(x,t)}( r_z(x,t)).
    \end{align*}
\end{cor}

\begin{proof}
    If $x \in \spann \{z\}$ then $R_z(x)=x$ and $R_{r_z(x,t)}( r_z(x,t)) = r_z(x,t)$ and the result holds.
    Suppose $x \notin \spann \{z\}$.
    By Theorem \ref{t:genrot}(\ref{l:genrot4}),
    $r_z(R_z(x),t)$ is not in the same connected component of $X \setminus \spann\{ r_z(z,t)\}$ as $r_z(x,t)$.
    The result follows from Lemma \ref{l:2point}.
\end{proof}

We next verify that
our definitions of reflections and rotations are consistent with those given for isometries.

\begin{lem}\label{l:notisom}
	For every linear isometry $g$ of an analytic normed plane $X$,
	there exists non-zero $z \in X$ and $t \in \mathbb{T}$ such that either $g(x) = r_z(x,t)$ for all $x \in X$,
	or $g(x) = r_z(R_z(x),t)$ for all $x \in X$.
\end{lem}

\begin{proof}
	Choose any pair of linearly independent points $y,z \in X$.
	The point $g(z)$ must lie somewhere on the circle $S_{\|z\|}[0]$
	thus, by Theorem \ref{t:genrot}, there exists a unique $t \in \mathbb{T}$ such that $g(z) = r_z(z,t)$.
	By Lemma \ref{l:2point}, Theorems \ref{t:genrefl} and \ref{t:genrot}, and Corollary \ref{c:switch},
	we have that $g(y)$ is equal to either $r_z(y,t)$ or $r_z(R_z(y),t)$.
	By Theorem \ref{t:genrot}(\ref{l:genrot4}), it follows that either $g = r_z(\cdot,t)$ or $g = r_z(R_z(\cdot),t))$.
\end{proof}

We now state a sufficient condition for determining whether a rotation, reflection, or composition of the two is an isometry.

\begin{lem}\label{l:theyareisom}
	Let $X$ be an analytic normed plane,
	and for some $z \in X$ and $t \in \mathbb{T}$ let $g \in \{r_z(\cdot,t),r_z(R_z(\cdot),t)\}$.
	If there exists an open set $U \subset X$ where $\|g(x)-g(y)\|=\|x-y\|$ for all $x,y \in X$,
	then $g$ is a linear isometry.
\end{lem}

\begin{proof}
	Let $\H,\H'$ be the open half-planes formed by the line $\ell$ through $0,z$.
	Suppose $\H \cap U \neq \emptyset$.
	As the set of points $(x,y) \in \H \times \H$ where $\|g(x)-g(y)\|=\|x-y\|$ is the zero set of an analytic function with connected domain $\H \times \H$,
	then $U=\H$ by Proposition \ref{p:anvar2}.	
	By Theorem \ref{t:mazurulam},
	there exists a unique linear isometry $h : X \rightarrow X$ where $h|_\H=g|_\H$,
	and by continuity we also have $h_\ell=g_\ell$.
	
	Choose any two points $x,y \in \H'$ and define $x' := z-x$ and $y' := z-y$.
	Since $\|g(x')\|= \|x'\|$ and $\|g(x') - g(z)\| = \|x'-z\|$ (and similarly for $y'$),
	it follows from Lemma \ref{l:2point} that either $g(x') = g(z) - g(x)$ or $g(x') = R_{g(z)}(g(z) - g(x))$.
	Since the latter would contradict Theorem \ref{t:genrot}(\ref{l:genrot4}),
	we have $g(x') = g(z) - g(x)$ and (since $x$ is arbitrary) $g(y') = g(z) - g(y)$.
	Hence
	\begin{eqnarray*}
		\|g(x) - g(y)\| &=& \| (g(z) - g(x)) - (g(z) - g(y)) \| \\
		&=& \|g(x') - g(y')\| \\
		&=& \|x'-y'\| \qquad \qquad \text{ (as $x',y' \in \H$)} \\
		&=& \|x-y\|.
	\end{eqnarray*}
	By Theorem \ref{t:mazurulam},
	there exists a unique linear isometry $h' : X \rightarrow X$ where $h'|_{\H'}=g|_{\H'}$;
	further,
	by continuity we note that $h'_\ell=g_\ell$.
	We also note that by Theorem \ref{t:genrot}(\ref{l:genrot4}),
	the sets $h(\H)$ and $h'(\H')$ cannot overlap.
	As both $h,h'$ agree on the line $\ell$, $h(\H) \cap h'(\H') = \emptyset$ and $g$ is bijective, it follows that $h=h' = g$.
\end{proof}

\begin{rem}
    It can be noted that all the results in this section except Theorem \ref{t:genrefl}(\ref{l:generalreflection2}) and Theorem \ref{t:genrot}(\ref{l:genrot2}) will still hold if we allow $X$ to be any smooth and strictly convex normed plane. 
    These two results can fail in such generality.
\end{rem}

\subsection{Normal points of reflections and rotations}

\begin{defn}\label{defn:nice}
	Let $K_4$ be the complete graph on the vertices $v_0,v_1,v_2,v_3$.
	For a normed plane $X$, a
	point $z \in X \setminus \{0\}$ is \emph{normal} if there exists a completely strongly regular framework $(K_4,p)$ with $p_{v_0} = 0$ and $p_{v_1} = z$.
	Given a normal point $z$,
	a point $y \in X \setminus \{0,z\}$ is \emph{$z$-normal} if there exists a completely strongly regular framework $(K_4,p)$ with $p_{v_0} = 0$, $p_{v_1} = z$ and $p_{v_2} = y$.
\end{defn}

Both normal and $z$-normal points have a slightly technical definition which will become important later on.
In analytic normed planes, almost all points will be normal.

\begin{prop}\label{prop:znice}
    Let $X$ be an analytic normed plane.
    Then the following holds.
    \begin{enumerate}[(i)]
        \item \label{prop:znice1} The set of normal points of $X$ is open and conull.
        \item \label{prop:znice2} For any normal point $z$,
        the set of $z$-normal points of $X$ is open and conull.
    \end{enumerate}
\end{prop}

\begin{proof}
    (\ref{prop:znice1}):
    By Proposition \ref{p:comreg},
    the set $\sreg (K_4;X)$ is an open conull set.
    As translations do not effect strong regularity,
    the set of all placements $p$ that correspond to a strongly regular framework $(K_4,p)$ with $p_{v_0}=0$ is an open conull subset of $\{0\} \times X^{\{v_1,v_2,v_3\}}$.
    The set of normal points is now a projection of this subset onto its $v_1$ coordinate.
    The result now follows as projections of open conull sets are open and conull.
    
    (\ref{prop:znice2}):
    Fix a normal point $z$ and a graph $G=(V,E) \cong K_4$ with distinct vertices $v_0,v_1,v_2$.
    Define the linear subspace $Z := \{ x \in X^V : x_{v_0}=0 , x_{v_1}=z\}$ and the analytic map $h_G := f_G|_{Z}^{\mathbb{R}^{E\setminus \{v_0 v_1\} }}$.
    By using similar methods to Proposition \ref{prop:regcreg} with $f_G$ replaced by $h_G$ and $X_G$ replaced with $Z \cap X_G$,
    we see that the set $Y := \{ x \in Z : x \text{ is a regular point of $h_G$} \}$ is an open conull subset of $Z$.
    Define $W := \{ x \in Z : h^{-1}_G(h_G(x)) \subset Y\}$.
    By Theorem \ref{thm:rigidne},
    $h_G|_{Y_G}$ is a submersion.
    By applying similar methods to Lemma \ref{l:sard2} and Proposition \ref{p:comreg},
    we see that $W$ is a conull subset of $Z$.
    The result now follows from Lemma \ref{l:infglobal1}.
\end{proof}

\begin{rem}
Proposition \ref{prop:znice}, which will suffice for our purposes, tells us that almost all points in any analytic normed space are normal. However we suspect that this can be strengthened to say that all non-zero points will be normal.
\end{rem}

The usefulness of normal points can be seen in the following two results.

\begin{thm}\label{l:sisom2}
	Let $X$ be an analytic normed plane,
	$z \in X$ be normal and $\H$ an open half plane formed by the line through $0$ and $z$.
	Then there exists a comeagre subset of points $y \in \H$ such that the following holds;
	for any $t \in \mathbb{T}$, 
	if $r_z(\cdot,t)$ is not an isometry then there exists an open conull subset of $\H$ of points $x$ where
	\begin{align*}
		\| r_z (x,t) - r_z(y,t) \| \neq \| x-y\|.
	\end{align*}
\end{thm}

\begin{thm}\label{l:srefl}
	Let $X$ be an analytic normed plane,
	$z \in X$ be normal and $\H$ an open half plane formed by the line through $0$ and $z$.
	Then there exists a comeagre subset of points $y \in \H$ such that the following holds;
	for any $t \in \mathbb{T}$,
	if $r_z(R_z(\cdot),t)$ is not an isometry then there exists an open conull subset of $\H$ of points $x$ where
	\begin{align*}
		\| r_z (R_z(x),t) - r_z(R_z(y),t) \| \neq \| x-y\|.
	\end{align*}
\end{thm}

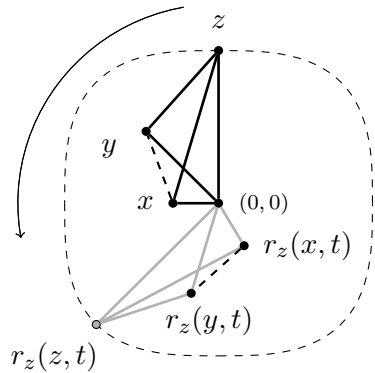
\begin{figure}[htp]
\begin{tikzpicture}[scale=1.2]

\draw [dashed,domain=0:90] plot ({1.687*(cos(\x))^(2/3)}, {1.687*(sin(\x))^(2/3)});
\draw [dashed,domain=90:180] plot ({-1.687*(-cos(\x))^(2/3)}, {1.687*(sin(\x))^(2/3)});
\draw [dashed,domain=180:270] plot ({-1.687*(-cos(\x))^(2/3)}, {-1.687*(-sin(\x))^(2/3)});
\draw [dashed,domain=270:360] plot ({1.687*(cos(\x))^(2/3)}, {-1.687*(-sin(\x))^(2/3)});

\node[vertex] (o) at (0,0) {};
\node[vertex] (z) at (0,1.687) {};
\node[vertex] (y) at (-0.794,0.794) {};
\node[vertex] (x) at (-0.5,0) {};

\draw[edge] (o)edge(z) {};
\draw[edge] (o)edge(x) {};
\draw[edge] (x)edge(z) {};
\draw[edge] (o)edge(y) {};
\draw[edge] (y)edge(z) {};
\draw[edge, thick, dashed] (x)edge(y) {};

\node (olabel) at (0.5,0) {\tiny$(0,0)$};
\node (zlabel) at (0, 2) {\small$z$};
\node (ylabel) at (-1.2,0.6) {\small$y$};
\node (xlabel) at (-.8,0) {\small$x$};

\draw[->,domain=100:190] plot ({2.2*cos(\x)}, {2.2*sin(\x)});

\node[fadedvertex] (z') at (-1.34,-1.34) {};
\node[vertex] (y') at (-0.3,-0.995) {};
\node[vertex] (x') at (0.28, -0.47) {};

\node (z'label) at (-1.8,-1.7) {\small$r_z(z,t)$};
\node (y'label) at (-0.1,-1.3) {\small$r_z(y,t)$};
\node (y'label) at (0.98, -0.47) {\small$r_z(x,t)$};

\draw[fadededge] (o)edge(z') {};
\draw[fadededge] (o)edge(x') {};
\draw[fadededge] (x')edge(z') {};
\draw[fadededge] (o)edge(y') {};
\draw[fadededge] (y')edge(z') {};
\draw[edge, thick, dashed] (x')edge(y') {};

\draw[->,domain=100:190] plot ({2.2*cos(\x)}, {2.2*sin(\x)});

\end{tikzpicture}
\caption{A $z$-rotation of two points $x,y$ in $\ell_4^2$. The distance between the points $x$ and $y$ will (for most choices of $x,y$) change during the $z$-rotation.}\label{fig:distchange}
\end{figure}

During a $z$-rotation, the distance between two points will usually change unless the rotation is an isometry; see for example Figure \ref{fig:distchange}.
What we will need to avoid are points $y$ where the $z$-rotation is not an isometry, 
but every point maintains their distance from $y$ during the $z$-rotation.
To see why the existence of these points are bad,
suppose that for all values of $y$ there exists $t_y \in \mathbb{T}$ where for any point $x$ we have $\|r_z(x,t_y) - r_z(y,t_y)\| = \|x-y\|$, but $r_z(\cdot,t_y)$ is not an isometry.
Fix $p$ to be a placement of the graph $K_5^-$ (see Figure \ref{fig:smallgraphs1}) in an analytic normed plane $X$ with degree 4 vertices $v_0, v_z, v_y$ at 0, $z$ and $y$ respectively.
Now define $q$ to be the placement of $K_5^-$ where $q_v = r_z(p_v,t_y)$ for each vertex $v$.
Then $(K_5^-,q)$ is equivalent to $(K_5^-,p)$,
but it is not congruent, as $r_z(\cdot,t_y)$ is not an isometry.
What Theorem \ref{l:sisom2} (and similarly Theorem \ref{l:srefl} for the $z$-reflection + $z$-rotation analogue) prove is that any such bad points like the previously mentioned $y$ form a meagre set,
hence we can essentially avoid them by some careful choices for our placements.

We shall defer the rather technical proofs required for Theorems \ref{l:sisom2} and \ref{l:srefl} until Section \ref{sec:tech}.

\section{Global rigidity in analytic normed planes}
\label{sec:global}

In this section we utilise our newly defined generalised reflections and rotations to prove multiple results regarding global rigidity in analytic normed planes.

\subsection{$K_5^-$ and $H$ are globally rigid}

We first need the following technical result.

\begin{lem}\label{l:semirefl}
	Let $X$ be an analytic normed plane, $z \in X$ be non-zero and $\H$ be one of the open half-planes formed by the line through $0,z$.
	Then for any $y \in \H$ and $t \in \mathbb{T}$,
	the set 
	\begin{align*}
		S(y,t) := \left\{ x \in \H :\| x-y\| \notin \{ \| r_z (x,t) - r_z(R_z(y),t) \|, \, \| r_z (R_z(x),t) - r_z(y,t) \| \}  \right\}
	\end{align*}
	is an open conull subset of $\H$.
\end{lem}

\begin{proof}
	Define the continuous function $f :\H \rightarrow \mathbb{R}$ where for each $x \in \H$ we have
	\begin{align*}
		f(x) := \left( \| r_z (x,t) - r_z(R_z(y),t) \| - \| x-y\|\right)\left( \| r_z (R_z(x),t) - r_z(y,t) \| - \| x-y\|\right).
	\end{align*}
	By Corollary \ref{c:switch} we have $r_z(R_z(y),t) = R_{r_z(y,t)}(r_z(y,t))$,
	and so by Theorem \ref{t:genrefl} we have $\| r_z (y,t) - r_z(R_z(y),t) \| \neq 0$.
	As $f(y) \neq 0$,
	the zero set of $f$ is a closed null set by Proposition \ref{p:anvar}.	
\end{proof}

We will now show that the graphs $K_5^-$ and $H$ depicted in Figure \ref{fig:smallgraphs1} are globally rigid in $X$. We will then deduce from the global rigidity of $K_5^-$ that all 3-laterations graphs (defined below) on at least 5 vertices are globally rigid.
One motivation for showing that $H$ is globally rigid in $X$ is that $H$ is the smallest globally rigid graph in $X$ that is not globally rigid in the Euclidean plane. In the Euclidean context one may use the 2-vertex-separation containing the two vertices of degree 5 to define a line to reflect one part of the graph through. We will give a second motivation for analysing these two specific graphs in the concluding remarks.

 \begin{figure}[htp]
\begin{center}
\begin{tikzpicture}[scale=.4]
\filldraw (-.5,0) circle (3pt)node[anchor=east]{};
\filldraw (0,3) circle (3pt)node[anchor=east]{};
\filldraw (3.5,0) circle (3pt)node[anchor=west]{};
\filldraw (3,3) circle (3pt)node[anchor=west]{};
\filldraw (1.5,-1.5) circle (3pt)node[anchor=west]{};

 \draw[black,thick]
(1.5,-1.5) -- (-.5,0) -- (0,3) -- (3.5,0) -- (3,3) -- (-.5,0) -- (3.5,0);

\draw[black,thick]
(1.5,-1.5) -- (0,3) -- (3,3) -- (1.5,-1.5);

        \end{tikzpicture}
          \hspace{0.5cm}
     \begin{tikzpicture}[scale=.4]
\filldraw (0,0) circle (3pt)node[anchor=east]{};
\filldraw (0,3.5) circle (3pt)node[anchor=east]{};
\filldraw (3.5,0) circle (3pt)node[anchor=north]{};
\filldraw (3.5,3.5) circle (3pt)node[anchor=south]{};
\filldraw (7,0) circle (3pt)node[anchor=west]{};
\filldraw (7,3.5) circle (3pt)node[anchor=west]{};

 \draw[black,thick]
(0,0) -- (0,3.5) -- (3.5,0) -- (3.5,3.5) -- (0,0) -- (3.5,0) -- (7,3.5);

\draw[black,thick]
(0,3.5) -- (3.5,3.5) -- (7,3.5) -- (7,0) -- (3.5,3.5);

\draw[black,thick]
(7,0) -- (3.5,0);

\end{tikzpicture}
\end{center}
\vspace{-0.3cm}
\caption{The graphs $K_5^-$ (left) and $H$ (right).}
\label{fig:smallgraphs1}
\end{figure}
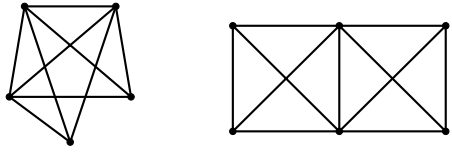

\begin{lem}\label{lem:k5-e}
	Let $X$ be an analytic normed plane and $G \cong K_5^-$ with vertex set $V=\{v_1, \ldots, v_5\}$ and missing edge $v_4 v_5$.
	Let $Q$ be the set of all frameworks $(G,p)$ where $p_{v_1}=0$, $p_{v_2} \neq 0$ and $p_{v_3},p_{v_4},p_{v_5}$ lie on the same side of the line through $p_{v_1}$ and $p_{v_2}$.
	Then $Q$ contains an open dense subset $Q'$ where for all $p \in Q'$,
	the framework $(G,p)$ is globally and infinitesimally rigid in $X$.
\end{lem}

\begin{proof}
Choose any placement $q \in Q$ such that:
 $q_{v_1}=0$ and $q_{v_2}$ is normal;
$q_{v_3}$ satisfies the properties of Theorems \ref{l:sisom2} and \ref{l:srefl} with $y= q_{v_3}$ and $z = q_{v_2}$; and $q$ is completely strongly regular (and hence also in general position).
Let $\H$ denote the open half-plane formed by the line through $q_{v_1}$ and $q_{v_2}$ containing $q_{v_3}$.
	The existence of a dense subset of such placements in $Q$ is guaranteed by Propositions \ref{p:comreg} and \ref{prop:znice} and Theorems \ref{l:sisom2} and \ref{l:srefl}.
	
	As $G$ contains a $(2,2)$-tight spanning subgraph then $(G,q)$ is infinitesimally rigid by Theorem \ref{thm:rigidne}.
	Define $B_\epsilon(q_{v_5})$ to be the open $\epsilon$-ball with center $q_{v_5}$,
	where $\epsilon >0$ is chosen so that $B_\epsilon(q_{v_5}) \subset \H$.
	It follows from Theorem \ref{thm:ave} that we only need to prove the following will hold for any sufficiently small $\epsilon > 0$ (i.e.~so that infinitesimal rigidity is preserved):
	there exists a dense subset of $B_\epsilon(q_{v_5})$ of points $x$ where if we set $p = (q_{v_1},q_{v_2},q_{v_3},q_{v_4},x)$ then the framework $(G,p)$ will be globally rigid.
	
	If $(G,q)$ is globally rigid then we may choose $p=q$ and we are done.
	Suppose otherwise.
	As $(G,q)$ is strongly regular and $X$ is non-Euclidean,
	there exists at most a finite $S \subset \mathbb{T}$ such that for each $t \in S$,
	both $r_z(\cdot,t)$ and $r_z( R_z(\cdot),t)$ are not isometries and one of the following holds:
	\begin{eqnarray*}
		\|r_z( q_{v_3},t) - r_z(q_{v_4},t)\| &=& \| q_{v_3} - q_{v_4}\|, \\
		\|r_z(q_{v_3},t) - r_z(R_z(q_{v_4}),t)\| &=& \| q_{v_3} - q_{v_4}\|, \\
		\|r_z( R_z(q_{v_3}),t) - r_z(q_{v_4},t)\| &=& \| q_{v_3} - q_{v_4}\|, \\
		\|r_z( R_z(q_{v_3}),t) - r_z(R_z(q_{v_4}),t)\| &=& \| q_{v_3} - q_{v_4}\|.
	\end{eqnarray*}
	By our choice of $q_{v_3}$ and Lemma \ref{l:semirefl},
	there exists a comeagre set of points $x \in \H$ such that for each $t \in S$,	
	\begin{eqnarray*}
		\|r_z( q_{v_3},t) - r_z(x,t)\| &\neq& \| q_{v_3} - x\|, \\
		\|r_z(q_{v_3},t) - r_z(R_z(x),t)\| &\neq& \| q_{v_3} - x\|, \\
		\|r_z( R_z(q_{v_3}),t) - r_z(x,t)\| &\neq& \| q_{v_3} - x\|, \\
		\|r_z( R_z(q_{v_3}),t) - r_z(R_z(x),t)\| &\neq& \| q_{v_3} - x\|,
	\end{eqnarray*}
	and $\| x - q_{v_5}\| < \epsilon$.
	We now define $p$ to be the placement of $G$ where $p_{v_i} := q_{v_i}$ for $i \in \{1,2,3,4\}$ and $q_{v_5} :=x$.
	Suppose $(G,p') \sim (G,p)$.
	By translating we may assume $p'_{v_1} = 0$,
	thus it follows that 
	\begin{align*}
		p'_{v_2} = r_z(p_{v_2},t) = r_z(z,t)
	\end{align*}
	for some $t \in \mathbb{T}$;
	further,
	we must have 
	\begin{align*}
		p'_{v_i} \in \{ r_z(p_{v_i},t) , r_z(R_z(p_{v_i}),t)\}
	\end{align*}
	for each $i=3,4,5$.
	However,
	by our choice of $x \in \H$,
	we note that none of these possibilities can hold unless $p' \sim p$.
	It follows that $(G,p)$ is globally rigid as required.
\end{proof}

Our next theorem now follows immediately from Lemma \ref{lem:k5-e}.

\begin{thm}\label{thm:k5-e}
	Let $X$ be an analytic normed plane.
	Then $K_5-e$ is globally rigid in $X$.
\end{thm}

We can also use a similar method to prove that $H$ is globally rigid.

\begin{thm}\label{thm:H}
	Let $X$ be an analytic normed plane.
	Then $H$ is globally rigid in $X$.
\end{thm}

\begin{proof}
	Label the vertices and edges of $H$ by $V:= \{v_1, \ldots, v_6\}$ and
	\begin{align*}
		E:= \{ v_i v_j : 1\leq i<j\leq 4\} \cup \{ v_i v_j : 3\leq i<j\leq 6\}
	\end{align*}
	respectively.
	Let $z$ be a normal point and let
	$\H$ be one of the open half-planes formed by the line through $0,z$.
 	Choose any $z$-normal point $y \in \H$ that satisfies Theorems \ref{l:sisom2} and \ref{l:srefl}.
	As $y$ is $z$-normal,
	we may choose $w \in \H$ such that the graph $K_4$ with placement $s:=(0,w,y,z)$ gives a strongly regular framework.
	Define $S \subset \mathbb{T}$ to be the set of values $t$ where one of the following holds:
	\begin{enumerate}[(i)]
		\item $r_z( \cdot, t)$ is not an isometry and $\| r_z(w,t) - r_z(y,t)\| = \|w-y\|$;
		\item $r_z( R_z(\cdot), t)$ is not an isometry and $\| r_z(R_z(w),t) - r_z(R_z(y),t)\| = \|w-y\|$;
		\item $\| r_z(w,t) - r_z(R_z(y),t)\| = \|w-y\|$;
		\item $\| r_z(R_z(w),t) - r_z(y,t)\| = \|w-y\|$.
	\end{enumerate}
	As $(K_4,s)$ is strongly regular and $K_4$ is rigid in $X$ (Theorem \ref{thm:rigidne}),
	then $S$ must be a finite set.
	By the properties of our choice of $y$ and Lemma \ref{l:semirefl},
	we may choose $x \in \H$ such that for each $t \in S$,
	\begin{eqnarray*}
		\|r_z( x,t) - r_z(y,t)\| &\neq& \|x-y\|, \\
		\|r_z(R_z(x),t) - r_z(R_z(y),t)\| &\neq& \|x-y\|, \\
		\|r_z( x,t) - r_z(R_z(y),t)\| &\neq& \|x-y\|, \\
		\|r_z( R_z(x),t) - r_z(x,t)\| &\neq& \|x-y\|
	\end{eqnarray*}
	and the graph $K_4$ with placement $s':=(0,x,y,z)$ is a strongly regular framework.
	Define $p$ to be the placement of $G$ with
	\begin{align*}
		p_{v_1}:=y, \quad p_{v_2} := w, \quad p_{v_3} := 0 , \quad p_{v_4} :=z, \quad p_{v_5} := y, \quad p_{v_6}:=x.
	\end{align*}
	As $(K_4,s)$ and $(K_4,s')$ are both regular and $K_4$ is rigid in $X$,
	the framework $(H,p)$ is infinitesimally rigid.
	Let $q$ be a placement of $H$ where $(H,q)$ is equivalent to $(H,p)$ and $q_{v_3}=0$.
	Then there exists $g_1,g_2,g_5,g_6 \in \{ r_z( \cdot ,t), r_z(R_z(\cdot),t)\}$ for some $t \in \mathbb{T}$ so that
	\begin{align*}
		q_{v_1} = g_1(y), \quad q_{v_2}= g_2(w), \quad q_{v_3}= 0, \quad q_{v_4} = r_z(z,t), \quad q_{v_5} = g_5(y), \quad q_{v_6} = g_6(x).
	\end{align*}
	If $q$ is not congruent to $p$ then we must have $t \in S$,
	however by our choice of $x$ it follows that
	\begin{align*}
		\|g_6(x) - g_5(y)\| \neq \|x-y\|,
	\end{align*}
	a contradiction.
	Hence $q$ is congruent to $p$ and $(H,p)$ is globally rigid.
	As $(H,p)$ is globally and infinitesimally rigid then by Theorem \ref{thm:ave},
	$H$ is globally rigid in $X$.
\end{proof}

\subsection{Some classes of globally rigid graphs}

We first note that we can always add a vertex of degree at least 3 to a graph to preserve global rigidity in analytic normed planes.

\begin{lem}\label{l:0+ext}
	Let $(G,p)$ be a globally rigid framework in an analytic normed plane $X$ and let $G'$ be the graph formed from $G$ by adding a vertex $w$ and connecting it by three edges to the vertices $v_1,v_2,v_3$.
	Suppose $p_{v_1},p_{v_2},p_{v_3}$ are not colinear.
	If we define for each $x \in X$ the placement $p^x$ of $G'$ with $p^x_v := p_v$ for all $v \in V(G)$ and $p^x := x$,
	the set of values $x$ where $(G',p^x)$ is globally rigid is an open conull subset of $X$.
\end{lem}

\begin{proof}
	For every distinct triple $i,j,k \in \{1,2,3\}$,
	define $\H_i$ to be the open half-plane containing $p_{v_i}$ formed by the line through $p_{v_j}$ and $p_{v_k}$.
	It is immediate that $\H_1 \cup \H_2 \cup \H_3$ is an open conull subset of $X$.
	It follows that we need only show that for each $i \in \{1,2,3\}$,
	there exists an open conull set of points $x \in \H_i$ where $(G',p^x)$ is globally rigid.
	
	Fix a distinct triple $i,j,k \in \{1,2,3\}$.
	By applying translations,
	we may assume that $p_{v_j}=0$.
	By setting $\H = \H_i$, $z=p_{v_k}$ and $y=p_{v_i}$,
	let $\H'$ be the open conull subset of points as described in Lemma \ref{l:semirefl} for $t=0$.
	Choose any $x \in \H'$ and let $(G',q)$ be an equivalent framework to $(G',p^x)$.
	Since $(G,p)$ is globally rigid we may assume by applying isometries that $q_v=p_v$ for all $v \neq w$.
	By Lemma \ref{l:2point},
	either $q_w =x$ or $q_w =R_{z}(x)$.
	As $\|R_{z}(x) - y\| \neq \|x-y\|$,
	we must have $q_w=x$,
	hence $(G',p^x)$ is globally rigid.
\end{proof}

With this we can now prove that almost all sufficiently large complete frameworks in analytic normed planes will be globally rigid.

\begin{thm}\label{t:complete}
	Let $X$ be an analytic normed plane and $n \geq 5$ a positive integer.
	Then there exists an open dense set of placements $p \in X^n$ where the framework $(K_n,p)$ is globally rigid in $X$.
\end{thm}

\begin{proof}
	It follows from Lemma \ref{l:0+ext} that it suffices to prove the result for $n=5$.
	Let $K_n$ be the complete graph on the vertex set $\{v_1,v_2,v_3,v_4,v_5\}$.
	By reordering the vertices of $K_5$ if necessary,
	any framework $(K_5,p)$ will contain a copy of $(K_5-v_4v_5,p)$ where $p_{v_3},p_{v_4},p_{v_5}$ lie on the same side of the line through $p_{v_1}$ and $p_{v_2}$.
	Hence the result now follows by applying Lemma \ref{lem:k5-e}.
\end{proof}

Recall that $G$ is a \emph{$k$-lateration graph} if it is obtained from $K_k$ by sequentially adding vertices of degree $k$ (we will make no assumption on the adjacencies of the neighbours of the new vertices).

\begin{thm}
\label{thm:3lat}
    Let $G$ be a 3-lateration graph on at least 5 vertices.
    Then $G$ is globally rigid in any analytic normed plane.
\end{thm}

\begin{proof}
    Every 3-lateration graph on at least 5 vertices can be formed from $K_5-e$ by a sequence of degree 3 vertex additions.
    The result now follows from Theorem \ref{thm:k5-e} and repeated application of Lemma \ref{l:0+ext}.
\end{proof}

Since edge addition preserves global rigidity it is clear that every graph containing a spanning 3-lateration subgraph is globally rigid in $X$. This immediately tells us that all complete graphs $K_n$, $n\geq 5$, are globally rigid in any analytic normed plane. Note that this is a weaker statement than we obtained in Theorem \ref{t:complete}.

We conjecture that if $n\geq 2d+1$ then $K_n$ is globally rigid in any $d$-dimensional analytic normed space $X$ and that $(d+1)$-lateration graphs on at least $2d+1$ vertices are globally rigid in $X$.

We can also apply the proof technique of Theorem \ref{thm:3lat} to $H$ to find another infinite family of globally rigid graphs. On the other hand, many graphs are not covered by these families. For example, these results do not confirm global rigidity for any bipartite graph, while it would be natural to expect that the complete bipartite graph $K_{m,n}$ is globally rigid whenever $m,n\geq 3$ and $m+n\geq 8$.

\section{Proof of Theorems \ref{l:sisom2} and \ref{l:srefl}}
\label{sec:tech}

\subsection{Set up and notation}

In this section we will prove Lemma \ref{l:sisom2}. The methods will then be modified to prove Lemma \ref{l:srefl}.
Throughout Section \ref{sec:tech} we fix the following: 
\begin{enumerate}[(i)]
	\item $X$ is an analytic normed plane.
	\item $z \in X$ is normal.
	\item $\H$ is an open half plane formed by the line through $0$ and $z$.
	\item $\H' \subset \H$ is the subset of $z$-normal points.
\end{enumerate}
Define for each $y \in \H$ the set $S_y$ of points $t \in \mathbb{T}$ where for all $x \in \H$,
\begin{align*}
	\| r_z (x,t) - r_z(y,t) \| = \| x-y\|,
\end{align*}
and we define for each $t \in \mathbb{T}$ the set 
\begin{align*}
	A_t := \{ y \in \H : t \in S_y \}.
\end{align*}
In the following sections we prove that, for almost all points $y$,
the set $S_y$ only contains values $t$ where $r_z(\cdot,t)$ is an isometry,
which in turn will prove Lemma \ref{l:srefl}.
The proof is structured as follows:
\begin{enumerate}[(i)]
    \item We first prove that if the set $A_t$ contains a smooth curve, then $r_z(\cdot,t)$ is an isometry (Lemma \ref{l:pathint2}).
    \item Next, we show that there exists an open dense set $\H'' \subset \H$ of points $y$ where perturbing $y$ will also just perturb the points in the set $S_y$, while keeping the number of points in $S_y$ constant.
    \item Finally we will prove there exists a countable set of points $W \subset \mathbb{T}$,
    where for each $t \in \mathbb{T} \setminus W$ we have either $A_t = \H$ and $r_z(\cdot,t)$ is an isometry, or $A_t = \emptyset$.
    We shall also prove that $\bigcup_{t \in W} A_t \cap \H''$ is a countable set.
\end{enumerate}

\subsection{Basic properties of \texorpdfstring{$S_y$}{Sy} }
We first consider the sets $S_y$.

\begin{lem}\label{l:nots}
	For all $y \in \H$ and any $t \in \mathbb{T} \setminus S_y$, 
	there exists an open conull subset of $\H$ of points $x$ where
	\begin{align*}
		\| r_z (x,t) - r_z(y,t) \| \neq \| x-y\|.
	\end{align*}
\end{lem}

\begin{proof}
	If we define the map
	\begin{align*}
		f : \H \setminus \{y\} \rightarrow \mathbb{R}, ~ x \mapsto \| r_z (x,t) - r_z(y,t) \| - \| x-y\|,
	\end{align*}
	then $f$ is an analytic function with an open connected domain.
	By Proposition \ref{p:anvar},
	the zero set of $f$ is either a closed null set or $\H \setminus \{y\}$.
	If the latter holds then $t \in S_y$,
	a contradiction,
	thus the former must hold as required.
\end{proof}

\begin{lem}\label{l:sfin}
	For any $y \in \H$,
	if $y$ is $z$-normal then $S_y$ is finite.
\end{lem}

\begin{proof}
	Suppose $S_y$ is not finite.
	Define for each $x \in \H$ the placement $p^x$ of $K_4$,
	with $p^x_{v_1}=0$,
	$p^x_{v_2}=z$,
	$p^x_{v_3} = y$,
	$p^x_{v_4}=x$.
	As $S_y$ is not finite then for each $x \in \H$ the configuration space of $(K_4,p^x)$ modulo isometries will also be infinite.
	Since $K_4$ is rigid in $X$ (by Theorem \ref{thm:rigidne}) any strongly regular framework $(K_4,p)$ will have a finite configuration space modulo isometries.
	Hence,
	$p^x$ is not strongly regular for any choice of $x \in \H$ and $y$ is not $z$-normal.
\end{proof}

\subsection{Structural properties of $A_t$}\label{sec:at}

In this section we shall prove that if $A_t$ contains a smooth curve then $r_z(\cdot,t)$ is an isometry (see Lemma \ref{l:apath}).
We first note the following.

\begin{lem}\label{l:ainterior}
	Let $t \in \mathbb{T}$.
	If the interior of $A_t$ is non-empty then $r_z( \cdot, t)$ is an isometry.
\end{lem}

\begin{proof}
	Let $U$ be an open connected set in $A_t$.
	For all $x,y \in U$ we have
	\begin{align*}
		\| r_z(x,t) - r_z(y,t)\| = \| x - y\|.
	\end{align*}
	Hence by Lemma \ref{l:theyareisom},
	$r_z( \cdot, t)$ is an isometry.
\end{proof}

\begin{lem}\label{l:distline}
	Let $a,b,c \in X$ with $\varphi_c(b) =0$ and $c \neq 0$,
	and define $\ell :=\{ a + tb : t \in \mathbb{R}\}$.
	Then for every $x \in X$ there exists a unique pair $x_b,x_c \in \mathbb{R}$ so that $x = a + x_b b + x_c c$ and $a+ x_b$ is the closest point in $\ell$ to $x$.
\end{lem}

\begin{proof}
	If $x \in \ell$ then the result clearly holds.
	Suppose $x \notin \ell$ and define the analytic function
	\begin{align*}
		f: \mathbb{R} \rightarrow \mathbb{R}_{\geq 0}, ~ t \mapsto \frac{1}{2}\| a + tb - x\|^2.
	\end{align*}
	Then, for any $t \in \mathbb{R}$, we have $f'(t) = \varphi_{a+tb-x}(b)$.
	As $x \notin \ell$,
	there exists a unique point $x_b \in \mathbb{R}$ such that $x- (a+x_b b) = x_c c$ for some $x_c \in \mathbb{R}$.
	Since $X$ is strictly convex (Proposition \ref{prop:strconv}) and $\varphi_c(b) =0$ then $f'(x_b)=0$, and $f'(t) \neq 0$ if $t \neq x_b$.
	We have $f(t) \rightarrow \infty$ as $t \rightarrow \pm \infty$,
	hence $x_b$ is the unique minimum of $f$ as required.
\end{proof}

For the following, the line segment $\{ ta+(1-t)b : t \in [0,1]\}$ is denoted by $[a,b]$, and the open line segment $\{ ta+(1-t)b : t \in (0,1)\}$ is denoted by $(a,b)$.

\begin{lem}\label{l:aline}
	Let $t \in \mathbb{T}$.
	If $A_t$ contains a line segment $[a,b]$ with $a \neq b$,
	then $r_z(\cdot,t)$ is an isometry.
\end{lem}

\begin{proof}
	Let $\ell$ be the line through $a,b$ and choose any $x \in \ell \cap \H$.
	Then $x = (1-\alpha) a +\alpha b$ for some $\alpha \in \mathbb{R}$,
	\begin{align*}
		\|x - a\| = |\alpha|\|b-a\| \mbox{ and } \|x - b\| = |1-\alpha|\|b-a\|.
	\end{align*}
	As $a,b \in A_t$ we have that 
	\begin{align*}
		\|r_z(x,t) - r_z(a,t)\| = \|x - a\| = |\alpha|\|b-a\| = |\alpha|\|r_z(b,t) - r_z(a,t)\| \mbox{ and} \\
		\|r_z(x,t) - r_z(b,t)\| = \|x - b\| = |1-\alpha|\|b-a\| = |1-\alpha|\|r_z(b,t) - r_z(a,t)\|.
	\end{align*}
	Since $X$ is strictly convex (Proposition \ref{prop:strconv}), Lemma \ref{l:2point} implies that $r_z(x,t) = (1-\alpha)r_z(a,t) + \alpha r_z(b,t)$. Thus the points of $r_z(\ell \cap \H,t)$ are colinear and the distances between points preserved.
	Fix a point $\ell^\perp \in X$ where $\varphi_{\ell^\perp}(b-a)=0$ and $\|\ell^\perp\|=1$.
	By Lemma \ref{l:distline} there exists, for each $x \in \H$, unique scalars $\alpha_x,\beta_x$ where: (i) $x = (1-\alpha_x) a +\alpha_x b + \beta_x \ell^\perp$,
	(ii) the closest point to $x$ in $\ell$ is $(1-\alpha_x) a +\alpha_x b$,
	(iii) the distance from $x$ to its closest point in $\ell$ is $|\beta_x|$,
	and (iv) $\varphi_{\ell^\perp} (x -a) = |\beta_x|$.
	The maps $x \mapsto \alpha_x,\beta_x$ are continuous,
	since the map that takes points to their closest point on a fixed convex set is continuous.
	
	Define $U$ to be the set of points $x$ in $\H$ where the closest point to $x$ in $\ell$ lies in the open interval $(a,b)$.
	By taking a suitable open cover of $(a,b)$ we see that $U$ contains an open neighbourhood of $(a,b)$.
	Thus we may choose an open connected subset $U' \subset U$.
	As the distance from any point to the line segment $[a,b]$ is preserved,
	the closest point to $r_z(x,t)$ in $r_z(\ell,t)$ is $r_z((1-\alpha_x) a +\alpha_x b,t)$ and the distance from $r_z(x,t)$ to $r_z((1-\alpha_x) a +\alpha_x b,t)$ is $|\beta_x|$.
	By Lemma \ref{l:distline} we may choose $w \in X$ so that $\varphi_{w}(r_z(b,t) - r_z(a,t))=0$, $\|w\|=1$ and for any $x \in U'$ we have
	\begin{align*}
		r_z(x,t) = r_z((1-\alpha_x) a +\alpha_x b ) \pm \beta_x w = (1-\alpha_x) r_z(a,t) +\alpha_x r_z(b,t) \pm \beta_x w.
	\end{align*}	
	By the continuity of $r_z (\cdot, t)$, we are able to choose $w$ such that the plus/minus sign above is always a plus sign.
	It follows that we may extend the map $r_z(\cdot ,t)|_{U'}$ to a unique affine map $T : X \rightarrow X$ by setting 
	\begin{align*}
		T((1-\alpha) a +\alpha b + \beta \ell^\perp) := (1-\alpha) r_z(a,t) +\alpha r_z(b,t) + \beta w
	\end{align*}
	for all $\alpha, \beta \in \mathbb{R}$.
	The map $T$ is a linear isometry,
	as $T(0)=0$ and $\|T(x)\| = \|x\|$ for all $x \in U'$ (see Theorem \ref{t:mazurulam}).
	Hence $r_z(\cdot,t)$ is an isometry by Lemma \ref{l:theyareisom}.
\end{proof}

\begin{lem}\label{l:pathint2}
	Let $O$ be an open subset of $\mathbb{R}^2$ and $\alpha:[0,1] \rightarrow O$ be a smooth path with $\alpha'(t) \neq 0$ for all $t \in [0,1]$.
	If $\alpha$ is not a line segment then there exists a non-empty open set $U \subset O \times O$ of pairs of distinct points $(x_1,x_2)$ where the line through both points intersects the path $\alpha$ in at least two distinct points.
\end{lem}

\begin{proof}
	Let $\cdot$ be the standard dot product of $\mathbb{R}^2$. Define the set
	\begin{align*}
		L:= \{ (n,d) \in \mathbb{R}^2 \times \mathbb{R} : n\cdot n=1 \}.
	\end{align*}
	Each element $(n,d)$ will represent the line $\{ x \in \mathbb{R}^2 : n \cdot x=d \}$,
	and the only other element that will represent the same line is $(-n,-d)$.
	Given the linear transform $T : \mathbb{R}^2 \rightarrow \mathbb{R}^2$ with $T(a,b) = (b,-a)$ (i.e.~the $90^\circ$ clockwise rotation),
	define the continuous map $l : O \times O \setminus \{(x,x) : x \in  O \} \rightarrow L$ with
	\begin{align*}
		l(x,y) = \left( \frac{T(x-y)}{\sqrt{T(x-y)\cdot T(x-y)}}, \frac{T(x-y)\cdot x}{\sqrt{T(x-y)\cdot T(x-y)}} \right).
	\end{align*}
	As $l$ is continuous and $\alpha$ is a curve in $O$,
	it suffices for us to show that there is an open set of elements $(n,d) \in L$ where $n \cdot \alpha(t) = d$ for at least two distinct values $t$ in $[0,1]$.
	
	As $\alpha$ is not a line segment and $\alpha' \neq 0$,
	there exists $0 \leq t_1<t_2<t_3\leq 1$ where $\alpha(t_1),\alpha(t_2),\alpha(t_3)$ are not colinear and $\alpha|_{[t_1,t_3]}$ is injective.
	As $\alpha(t_1),\alpha(t_2),\alpha(t_3)$ are not colinear,
	there exists an open set $U \subset L$ of points $(n,d)$ where $n\cdot \alpha(t_1)< d$, $n\cdot \alpha(t_2) > d$ and $n\cdot \alpha(t_3)<d$.
	By the intermediate value theorem,
	for each $(n,d) \in U$ there exists $s_1 \in (t_1,t_2)$ and $s_2 \in (t_2,t_3)$ where $n\cdot \alpha(s_1)= n\cdot \alpha(s_2)=d$ as required.	
\end{proof}

We are now ready to prove our first key lemma of the section.

\begin{lem}\label{l:apath}
	For all $t \in \mathbb{T}$,
	if $A_t$ contains a smooth path	then $r_z( \cdot, t)$ is an isometry.
\end{lem}

\begin{proof}
	Let $\alpha :[0,1] \rightarrow \mathbb{R}^2$ be a smooth path in $A_t$.
	If $\alpha$ is a line segment then the result holds by Lemma \ref{l:aline}.
	Suppose otherwise.
	By Lemma \ref{l:pathint2},
	there exists an open set $U \subset \H \times \H$ of points $(x_1,x_2)$ where the line through $x_1,x_2$ intersects two distinct of $\alpha$.
	Choose any $(x_1,x_2) \in U$ and suppose they intersect $\alpha$ at the distinct points $\alpha(c_1),\alpha(c_2)$.
	Since the points are colinear and the distances between the pairs $\{x_1,\alpha(c_1)\}$, 
	$\{x_1,\alpha(c_2)\}$,
	$\{x_2,\alpha(c_1)\}$,
	$\{x_2,\alpha(c_2)\}$ and $\{\alpha(c_1),\alpha(c_2)\}$ are preserved under $r_z(\cdot ,t)$,
	then it follows from Lemma \ref{l:2point} that
	\begin{align*}
		\|r_z(x_1,t) - r_z(x_2,t)\| = \|x_1 - x_2\|
	\end{align*}
	for each $(x_1,x_2) \in U$.
	Define the function
	\begin{align*}
		f: \H \times \H \rightarrow \mathbb{R}, ~ (x_1,x_2) \mapsto \|r_z(x_1,t) - r_z(x_2,t)\| - \|x_1 - x_2\|.
	\end{align*}
	The map $f$ is analytic on the open connected set $\{ (x_1,x_2) \in \H \times \H : x_1 \neq x_2\}$.
	Hence,
	as the open set $U$ is contained in the zero set of $f$ then $f (x_1,x_2) =0$ for all $(x_1,x_2) \in \H$.
	By Lemma \ref{l:theyareisom},
	$r_z(\cdot,t)$ is an isometry.
\end{proof}

\subsection{Determining how \texorpdfstring{$S_y$}{Sy} changes under perturbation}

In this section we will construct the tools to help us understand how $S_y$ changes as $y$ moves (see Lemma \ref{l:hcover}).
We recall that $d_{\mathbb{T}}$ is the metric of the circle group $\mathbb{T}$ (see Section \ref{sec:prelim}).

\begin{lem}\label{l:bandiff}
	Let $C^1(\mathbb{T})$ be the Banach space of continuously differentiable functions from $\mathbb{T}$ to $\mathbb{R}$ with the norm 
	\begin{align*}
		\| f \|' := \sup_{t \in \mathbb{T}} | f (t)| + \sup_{t \in \mathbb{T}} | f' (t)|.
	\end{align*}
	Suppose $f \in C^1 (\mathbb{T})$ has a finite zero set $Z(f)$,
	and for every $t \in \mathbb{T}$,
	$f'(t) \neq 0$.
	Then for every $\epsilon >0$ where
	\begin{align*}
		\epsilon < \frac{1}{2} \min \{ d_{\mathbb{T}}(s,t) : s, t \in Z(f), s \neq t\} 
	\end{align*}
	there exists $\delta > 0$ such that if $\|g -f \|' < \delta$ then $|Z(g)| = |Z(f)|$ and $g'(t) \neq 0$ for all $t \in Z(g)$,
	and for every $t \in Z(f)$ there exists a unique $t' \in Z(g)$ where $d_{\mathbb{T}}(t,t')< \epsilon$.
\end{lem}

\begin{proof}
	Fix $\epsilon$ and $f \in C^1(\mathbb{T})$.
	As the set $Z(f')$ is compact and disjoint from $Z(f)$,
	we may choose $r > 0$ so that $r < |f(t)|$ for all $t \in \mathbb{T}$ where $f'(t) = 0$.	Define for each $s \in Z(f)$ the set $I_s$ to be the largest closed connected set containing $s$ where $|f(t)| \leq r$ for all $t \in I_s$.
	By shrinking $r$ we also shrink each set $I_s$,
	hence we may assume $r$ was chosen to be small enough that the length of the each of the intervals is less than $\epsilon$.
	Hence,
	$I_s \cap I_{s'} = \emptyset$ for all distinct $s,s' \in Z(f)$.
	For each $s \in Z(f)$, the restricted function $f|_{I_s}$ is strictly monotonic, and there exists $\delta_s >0$ where $|\delta_s| < f'(t)$ for all $t \in I_s$.
	Furthermore,
	for each $s \in Z(f)$ we can label the two boundary points of $I_s$ as $s_{\min},s_{\max}$ so that $f(s_{\min}) = \min_{t \in I_s} f(t) <0$ and $f(s_{\max}) = \max_{t \in I_s} f(t) >0$.
	
	Choose any $0< \delta < r$ so that for every $s \in Z(f)$ we have $\delta < \min \{|f(s_{\min})|,|f(s_{\max})| \}$ and $\delta < \delta_s$.
	If we choose any $g \in C^1(\mathbb{T})$ where $\|g-f\|' < \delta$,
	then the following will hold:
	\begin{enumerate}[(i)]
		\item For every $t \in \mathbb{T} \setminus \bigcup_{s \in Z(f)} I_s$,
		we have
		\begin{align*}
			|g(t)| \geq |f(t)| - |g(t)-f(t)| > r - \delta >0,
		\end{align*}
		thus $g$ has no zeroes in $\mathbb{T} \setminus \bigcup_{s \in Z(f)} I_s$.
		\item For every $s \in Z(f)$ and every $t \in I_s$,
		we have
		\begin{align*}
			|g'(t)| \geq |f'(t)| - |g'(t)-f'(t)| > \delta_s - \delta >0,
		\end{align*}
		thus each function $g|_{I_s}$ is strictly monotonic and $g'(s) \neq 0$.
		\item For each $s \in Z(f)$ we have
		\begin{align*}
			g(s_{\min}) < f(s_{\min}) +\delta < 0 \\
			g(s_{\max}) > f(s_{\max}) - \delta >0
		\end{align*}
		as $f(s_{\min}) < -\delta < 0 < \delta < f(s_{\max})$ by our choice of $\delta$.
		Combined with $g|_{I_s}$ being strictly monotonic,
		this implies there is exactly one zero of $g$ in the interval $I_s$.
	\end{enumerate}
	Since each of the intervals has width less than $\epsilon$,
	the result now holds.	
\end{proof}

\begin{lem}\label{l:sinfgood}
	For all $y \in \H'$,
	there exists an open conull subset $R \subset \H$ of points where the following holds for all $x \in R$;
	there exist only finitely many values $t \in \mathbb{T}$ where
	\begin{align*}
		\| r_z (x,t) - r_z(y,t) \| = \| x-y\|,
	\end{align*}
	and when the above does hold we have
	\begin{align*}
		\frac{d}{dt}\left(\| r_z (x,t) - r_z(y,t) \| - \| x-y\| \right) = \frac{d}{dt}\left(\| r_z (x,t) - r_z(y,t) \| \right) \neq 0.
	\end{align*}
\end{lem}

\begin{proof}
	For every $x \in \H$,
	let $(K_4,p^x)$ be the framework defined in Lemma \ref{l:sfin}.
	By employing similar methods to those outlined in the proof of Proposition \ref{prop:znice}(\ref{prop:znice2}),
	we see that there exists an open conull set $R$ of points $x$ where $(K_4,p^x)$ is completely strongly regular and infinitesimally rigid.
	Define 
	\begin{align*}
		p^x(t) := (r_z(p_{v_i}^x,t))_{i=1}^4.
	\end{align*}
	If $x \in R$ then for each $t \in \mathbb{T}$ the framework $(K_4 - v_3 v_4,p^x(t))$ has a single infinitesimal flex that is $0$ at $v_1$ (modulo scaling),
	and that flex is exactly 
	\begin{align*}
		u_t := (\frac{d}{dt} r_z(p_{v_i}^x,t))_{i=1}^4.
	\end{align*}
	Further,
	$(K_4,p^x(t)) \sim (K_4,p^x)$ if and only if 
	\begin{align}\label{e:sinfgood}
		\| r_z (x,t) - r_z(y,t) \| = \| x-y\|.
	\end{align}
	If $x \in R$ it follows that there exist only finitely many equivalent frameworks to $(K_4,p)$ with the vertex $v_1$ fixed at $0$,
	thus only finitely many $t \in \mathbb{T}$ can satisfy Equation (\ref{e:sinfgood}).	
	Choose any $x \in R$ and $t \in \mathbb{T}$ such that equation \ref{e:sinfgood} holds.
	Then the flex $u_t$ is a non-trivial flex of $(K_4,p^x(t))$ if and only if 
	\begin{align*}
		\frac{d}{dt}\left(\| r_z (x,t) - r_z(y,t) \| \right) =0.
	\end{align*}
	As $(K_4,p^x(t))$ is infinitesimally rigid,
	the flex $u_t$ is trivial as required.
\end{proof}

We now define $\H''$ to be the dense subset of $z$-normal points $y$ where for some neighbourhood $U$ of $y$ we have $|S_y| \leq |S_x|$ for all $x \in U$.
The following key result will allow us to determine how the $S_y$ sets continuously change as we move $y$ in $\H''$.

\begin{lem}\label{l:hcover1}
	For all $y \in \H''$ there exists an open neighbourhood $U \subset \H$ of $y$ and points $a,b \in \H \setminus U$ such that the following holds:
	\begin{enumerate}[(i)]
		\item For each $x \in U$,
		a point $t \in \mathbb{T}$ lies in $S_x$ if and only if
		\begin{align*}
			\| r_z (a,t) - r_z(x,t) \| = \| a-x\|, \qquad \| r_z (b,t) - r_z(x,t) \| = \| b-x\|.
		\end{align*}
		\item There exists $k \in \mathbb{N}$ such that $|S_x| = k$ for all $x \in U$ (and hence $U \subset \H'$).
		\item There exists continuous functions $f_1, \ldots, f_{k} : U \rightarrow \mathbb{T}$ such that for each $x \in U$,
		\begin{align*}
			0 = f_1(x) < \ldots < f_{k}(x) < 2 \pi
		\end{align*}
		are exactly the elements of $S_x$.		
	\end{enumerate}
\end{lem}

\begin{proof}
	Define $k := |S_y|$ and order the elements of $S_y$ as $0=: t_1 < \ldots < t_k < 2 \pi$.
	Following from Lemma \ref{l:sinfgood},
	we define $O$ to be the conull set of points $x \in \H\setminus \{y\}$ where the set
	\begin{align*}
		O_x := \{ t \in \mathbb{T} : \| r_z (x,t) - r_z(y,t) \| = \| x-y\| \}
	\end{align*}
	is finite,
	and if $t \in O_x$ then
	\begin{align*}
		\frac{d}{dt}\left(\| r_z (x,t) - r_z(y,t) \| - \| x-y\| \right) = \frac{d}{dt}\left(\| r_z (x,t) - r_z(y,t) \| \right) \neq 0.
	\end{align*}	
	It is immediate that $S_y \subset O_x$ for all $x \in O$.
	Choose and fix $a \in O$.
	As $O_a \setminus S_y$ is finite then by Lemma \ref{l:nots},
	there is a conull subset of points $O' \subset O$ of points $x$ where
	\begin{align*}
		\| r_z (x,s) - r_z(y,s) \| \neq \| x-y\|
	\end{align*}
	for each $s \in O_a \setminus S_y$.
	Now choose and fix $b \in O'$.
	For our choice of $a,b \in \H$,
	the following will hold:
	\begin{enumerate}[(i)]
		\item $a,b \neq y$.
		\item Any point $t \in \mathbb{T}$ will lie in $S_y$ if and only if $\| r_z (a,t) - r_z(y,t) \| = \| a-y\|$ and $\| r_z (b,t) - r_z(y,t) \| = \| b-y\|$.
		\item If $t \in \mathbb{T}$,
		$x \in \{a,b\}$ and $\| r_z (x,t) - r_z(y,t) \| = \| x-y\|$, then 
		\begin{align*}
			\frac{d}{dt}\left(\| r_z (x,t) - r_z(y,t) \| - \| x-y\| \right) = \frac{d}{dt}\left(\| r_z (x,t) - r_z(y,t) \| \right) \neq 0.
		\end{align*}
	\end{enumerate}
	
	Define for every $x \in \H $ the analytic maps $g^a_{x}, g^b_x : \mathbb{T} \rightarrow \mathbb{R}$,
	where for all $t \in \mathbb{T}$ we have
	\begin{align*}
		g^a_{x} (t) &:=& \| r_z (a,t) - r_z(x,t) \| - \| a-x\|, \qquad
		g^b_{x} (t) &:=& \| r_z (b,t) - r_z(x,t) \| - \| b-x\|.
	\end{align*}
	Also define the continuous maps
	\begin{align*}
		g^a : \H \rightarrow C^1(\mathbb{T}), ~ x \mapsto g_{x}, \qquad g^b : \H \rightarrow C^1(\mathbb{T}), ~ x \mapsto g_{x},
	\end{align*}
	where $C^1(\mathbb{T})$ is the Banach space described in Lemma \ref{l:bandiff}.
	By definition we note that $S_y = Z(g^a_y) \cap Z(g^b_y)$.
	Choose any $\epsilon >0$ so that
	\begin{eqnarray*}
		\epsilon &<& \frac{1}{2} \min \left\{ |s-t| : s, t \in Z(g^a_y), ~ s \neq t\right\},\\
		\epsilon &<& \frac{1}{2} \min \left\{ |s-t| : s, t \in Z(g^b_y), ~ s \neq t\right\},\\
		\epsilon &<& \frac{1}{2} \min \left\{ |s-t| : s \in Z(g^a_y) \setminus Z(g^b_y), ~ t \in Z(g^b_y) \setminus Z(g^a_y) \right\}.
	\end{eqnarray*}
	By the continuity of $g^a$ at $g^a_y$ and $g^b$ at $g^b$, combined with Lemma \ref{l:bandiff},
	there exists $\delta > 0$ such that the following holds for all $x \in \H$ where $\|x - y \| < \delta$:	
	\begin{enumerate}[(i)]
		\item $\|a - y \| > \delta$ and $\|b - y \| > \delta$.
		\item $|Z(g^a_x)| = |Z(g^a_y)|$ and for every $t \in Z(g^a_y)$ there exists a unique $s \in Z(g^a_x)$ so that $d_{\mathbb{T}}(s,t)<\epsilon$.
		\item $|Z(g^b_x)| = |Z(g^b_y)|$ and for every $t \in Z(g^b_y)$ there exists a unique $s \in Z(g^b_x)$ so that $d_{\mathbb{T}}(s,t)<\epsilon$.
	\end{enumerate}
	By our choice of $\epsilon$ we must have that $|S_x| \leq |Z(g^a_x) \cap Z(g^b_x)| \leq k$ if $\|x -y\|<\delta$.
	As $y \in \H'$,
	it follows that we can choose an open neighbourhood $U$ of $y$ where for all $x \in U$ we have that (i) $\|x-y\|<\delta$ and (ii) $k \leq |S_x|$.
	This in turn implies $|S_x|=k$ for all $x \in U$;
	importantly, $|S_x|=k$ (and hence $S_x= Z(g^a_x) \cap Z(g^b_x)$) for all $x \in U$.
	Since both $g^a$ and $g^b$ are continuous,
	we can now apply Lemma \ref{l:bandiff} to each pair of maps $g_x^a,g^b_x$ for every $x \in U$ to obtain the continuous functions $f_1,\ldots,f_k$ as required.
\end{proof}	

We now obtain the following result by applying Lemma \ref{l:hcover1} to each point in $\H''$.
	
\begin{lem}\label{l:hcover}
	There exists a countable open cover $\{U_i \subset \H : i \in I\}$ of $\H''$ and a set of pairs $((a_i,b_i))_{i \in I}$ with $a_i,b_i \in \H \setminus U_i$ such that for each $i \in I$ the following holds:
	\begin{enumerate}[(i)]
		\item For each $y \in U_i$,
		a point $t \in \mathbb{T}$ lies in $S_y$ if and only if
		\begin{align*}
			\| r_z (a_i,t) - r_z(y,t) \| = \| a_i-y\|, \qquad \| r_z (b_i,t) - r_z(y,t) \| = \| b_i-y\|.
		\end{align*}
		\item There exists $k_i \in \mathbb{N}$ such that $|S_y| = k_i$ for all $y \in U_i$.
		\item There exists continuous functions $f_1, \ldots, f_{k_i} : U_i \rightarrow \mathbb{T}$ such that for each $y \in U_i$,
		\begin{align*}
			0 = f_1(y) < \ldots < f_{k_i}(y) < 2 \pi
		\end{align*}
		are exactly the elements of $S_y$.		
	\end{enumerate}
\end{lem}

\begin{proof}	
	Apply Lemma \ref{l:hcover1} to every point $y \in \H''$ to obtain a corresponding set open neighbourhood $U_y$ of $y$,
	a pair of points $\{a_y,b_y\}$ with $a_y,b_y \notin U_y$,
	and some $k_y \in \mathbb{N}$.
	The set $\{ U_y :y \in \H''\}$ defines an open cover of $\H''$.
	The set $\H''$ is open by Lemma \ref{l:hcover1}.
	Hence $\H''$ is locally compact,
	and there exists a countable subcover $\{U'_i :i \in I\}$.
\end{proof}

It now follows that the set $\H''$ is an open dense subset of $\H$.

\subsection{Proof of Theorems \ref{l:sisom2} and \ref{l:srefl}}\label{sec:proof}

We are now finally ready to tackle the following key lemma.

\begin{lem}\label{l:aisom}
	There exists a countable set $W \subset \mathbb{T}$ such that the following holds:
	\begin{enumerate}[(i)]
		\item If $t \in \mathbb{T} \setminus W$ then either $A_t =\H$ and $r_z(\cdot, t)$ is an isometry,
		or $A_t \cap \H'' = \emptyset$.
		\item If $t \in W$ then $A_t \cap \H''$ is a countable set.
	\end{enumerate}
\end{lem}

\begin{proof}
	Let $\{U_i : i \in I\}$ be the countable open cover of $\H''$ with corresponding points $\{(a_i,b_i) : i \in I\}$ as defined in Lemma \ref{l:hcover};
	by separating any disconnected open sets into their countable components,
	we may assume each $U_i$ is connected.
	We note that it is suffice to show that the result holds for $A_t \cap U_i$ (given arbitrary $i \in I$) instead of for $A_t \cap \H''$,
	as the former will imply the latter.
	Due to this,
	we shall fix $i \in I$, we shall also let $k_i \in \mathbb{N}$ and $f_1, \ldots, f_{k_i}$ be as described in Lemma \ref{l:hcover}.
	
	Consider the ordering $0 \leq t < 2 \pi$ for $\mathbb{T}$;
	if the maximum of a set of points is $2 \pi$, we shall consider it to be 0.
	Define $W_i$ to be the finite set of points $t \in \mathbb{T}$ where the following holds:
	\begin{enumerate}
		\item There exists $j' \in \{1, \ldots, k_i\}$ so that either $\min_{x \in U_i} f_{j'} (x) = t$ or $\max_{x \in U_i} f_{j'} (x) = t$.
		\item For all other $j \in \{1, \ldots, k_i\}$ we have either $\min_{x \in U_i} f_j (x) = t$,
		$\max_{x \in U_i} f_j (x) = t$,
		or $f_j (x) \neq t$ for all $x \in U_i$.	
	\end{enumerate}
	Note that we are using minimums and maximums that may or may not exists,
	so $W_i$ may be empty.
	Define for each $t \in \mathbb{T}$ the analytic map $g_t : U_i \rightarrow \mathbb{R}$,
	where for $x \in U_i$ we have
	\begin{align*}
		g_t (x) = \left( \|r_z(x,t) - r_z(a_i,t)\| - \|x-a_i\| \right)^2 + \left( \|r_z(x,t) - r_z(b_i,t)\| - \|x-b_i\| \right)^2.
	\end{align*}
	By Lemma \ref{l:hcover} we have that $g_t(x) = 0$ if and only if $t \in S_x$,
	hence	$Z(g_t) = A_t \cap U_i$ for each $t \in \mathbb{T}$.	
	We now need to show three things hold for any $t \in \mathbb{T}$:
	\begin{enumerate}[(i)]
		\item \label{e:aisom1} The set $A_t \cap U_i$ is either countable or $A_t \cap U_i = U_i$.
		\item \label{e:aisom2} If $A_t \cap U_i = U_i$ then $r_z(\cdot, t)$ is an isometry.
		\item \label{e:aisom3} If $t \notin W_i$ and $A_t \cap U_i$ is countable then $A_t \cap U_i = \emptyset$.
	\end{enumerate}
	
	(\ref{e:aisom1}): 
	By Proposition \ref{p:anvar2},
	we have that either $A_t \cap U_i \in \{\emptyset, U_i\}$,
	$A_t \cap U_i$ is countable,
	or there exists an injective analytic path $\phi : [0,1] \rightarrow A_t \cap U_t$ with $\phi' \neq 0$.
	If the latter holds then by Lemma \ref{l:apath},
	$A_t \cap U_i = U_i$ as required.
	
	(\ref{e:aisom2}): This follows immediately from Lemma \ref{l:ainterior}.
	
	(\ref{e:aisom3}): Suppose $A_t \cap U_i$ is countable but not empty for some $t \notin W_i$.
	By our assumptions plus the continuity of the maps $f_1, \ldots, f_{k_i}$,
	there exists $j \in \{1,\ldots, k_i\}$ and $x_0,x_1 \in U_i$ such that $f_j(x_0) < t< f_j(x_1)$.
	We note that for any continuous path $\alpha:[0,1] \rightarrow U_i$ with $\alpha(0) = x_0$ and $\alpha(1) = x_1$,
	we have by the Intermediate Value Theorem the existence of a point $c \in (0,1)$ where $f_j(\alpha(c)) = t$.
	Since we can define an uncountable amount of distinct continuous paths from $x_0$ to $x_1$,
	$A_t \cap U_i$ cannot be countable,
	a contradiction.	
	
	As $i \in I$ was chosen arbitrarily,
	(\ref{e:aisom1}), (\ref{e:aisom2}) and (\ref{e:aisom3}) all hold for any $i \in I$.
	The required result now follows immediately with $W := \bigcup_{i \in I} W_i$.
\end{proof}

This allows us to prove our final key lemma.

\begin{lem}\label{l:sisom}
	There is a comeagre set of points $y \in \H$ with $S_y$ being exactly the set of points $t$ where $r_z(\cdot, t)$ is an isometry.
\end{lem}

\begin{proof}
	Let $W \subset \mathbb{T}$ be the set defined in Lemma \ref{l:aisom}.
	If we define $A := \bigcup_{t \in W} A_t$ then $A \cap \H''$ is a countable set.
	Choose any $y \in \H'' \setminus A$ and note that $\H'' \setminus A$ is a comeagre subset of $\H$.
	As $S_y = \{ t \in \mathbb{T} : y \in A_t\}$ then $t \in S_y$ if and only if $r_z(\cdot, t)$ is an isometry.	
\end{proof}

\begin{proof}[Proof of Theorem \ref{l:sisom2}]
	This follows immediately from Lemma \ref{l:nots} and Lemma \ref{l:sisom}.
\end{proof}

\begin{proof}[Proof of Theorem \ref{l:srefl}]
	By redefining $S_y$ to be the set of points $t \in \mathbb{T}$ where for all $x \in \H$,
	\begin{align*}
		\| r_z (R_z(x),t) - r_z(R_z(y),t) \| = \| x-y\|.
	\end{align*}
	We note that the methods of Section \ref{sec:tech} can be adapted easily to prove the result.
\end{proof}

\section{Concluding remarks}
\label{sec:conc}

\noindent \textbf{1. Global rigidity.} We aim to use the results of this article to prove a characterisation of global rigidity in analytic normed planes. This will be analogous to the following key result in combinatorial rigidity theory for the Euclidean plane.

\begin{thm}\cite{C2005,hendrickson,J&J}\label{thm:euclidglobal}
	A generic bar-joint framework $(G,p)$ in the Euclidean plane is globally rigid if and only if $G$ is either a complete graph on at most 3 vertices or $G$ is 3-connected and redundantly rigid.
\end{thm}

Indeed we conjecture that in any analytic normed plane a completely regular bar-joint framework is globally rigid if and only if it is 2-connected and redundantly rigid. Note that it is relatively straightforward to deduce that 2-connectivity is a necessary condition and the fact that $H$ is globally rigid (Theorem \ref{thm:H}) shows that 3-connectivity is not necessary. However it is a substantial task to show that every 2-connected and redundantly rigid graph is globally rigid in an analytic normed plane. Note that, by Theorem \ref{thm:22normed} and \cite[Theorem 5.3]{NOP}, the combinatorial conditions of redundant rigidity and 2-connectivity are exactly the same as those that arise in the study of global rigidity for frameworks on the cylinder in $\mathbb{R}^3$ \cite{JNglobal}. This gives additional motivation for our small graph results since $K_5^-$ and $H$ are exactly the base graphs used in the recursive proof of that work.\\

\noindent \textbf{2. Normed planes.} We conclude by commenting on the generality our results are proved in. We expect that one can adapt our main global rigidity results to replace analytic normed plane with any smooth $\ell_q$-norm, $1<q<\infty$. The key technical step required to achieve this would be to adapt Lemma \ref{l:sisom2}. On the other hand to analyse global rigidity in $\ell_1$, $\ell_\infty$ or general polyhedral norms seems to require completely different techniques. See \cite{kit} for an analysis of rigidity in polyhedral norms in terms of coloured sub-frameworks. It would be interesting to extend that analysis to global rigidity. The reader unfamiliar with the standard Euclidean rigidity theory may wonder about higher dimensional normed spaces. Unfortunately it is a long-standing open problem to understand rigidity (and global rigidity) in $d$-dimensional Euclidean space. It is not obvious that other normed spaces would be any simpler.\\

\noindent \textbf{3. Tingley's problem.} 
Tingley's problem (originally proposed by Tingley in \cite{tingley}) asks the following question: given two Banach spaces $X,Y$ with unit spheres $S_X,S_Y$, will any surjective isometry $f : S_X \rightarrow S_Y$ extend to a linear isometry $\tilde{f} :X \rightarrow Y$?
The question seems to be difficult and is currently still unknown even in the case where $X$ is a normed plane and $Y=X$.
We note that it is actually sufficient in this case to prove the following:
for every normed plane $X$, there exists a globally rigid framework $(K_n,p)$ in $X$ for some $n \geq 5$ where $p_{v_0} = 0$ for some $v_0 \in V(K_n)$ and $\|p_v\|=1$ for all $v \in V \setminus \{v_0\}$.
The idea to prove this in the case where $X$ is an analytic normed plane would be to develop analogues of Theorems \ref{l:sisom2} and \ref{l:srefl} for rotations restricted to the unit circle. If this can be done, it should be relatively easy to adapt Lemma \ref{lem:k5-e} to prove that any surjective isometry $f : S_X \rightarrow S_X$ can be extended to a linear isometry when $X$ is an analytic normed plane.

\subsection*{Acknowledgements}
Parts of this research was completed during the Fields Institute Thematic Program on Geometric constraint
systems, framework rigidity, and distance geometry. SD was partially supported by Austrian Science Fund (FWF): P31888.
AN was partially supported by the Heilbronn institute for mathematical research.

\bibliographystyle{abbrv}
\def\lfhook#1{\setbox0=\hbox{#1}{\ooalign{\hidewidth
  \lower1.5ex\hbox{'}\hidewidth\crcr\unhbox0}}}

\end{document}